\documentclass[11pt]{article}%
\usepackage{amssymb,amsmath,amsfonts,amsthm,array,bm,bbm,color}%
\usepackage{mathrsfs}%
\usepackage{graphicx}
\providecommand{\U}[1]{\protect\rule{.1in}{.1in}}

\numberwithin{equation}{section}

\newtheorem{theorem}{Theorem}[section]
\newtheorem{lemma}[theorem]{Lemma}

\newtheorem{proposition}[theorem]{Proposition}
\newtheorem{remark}[theorem]{Remark}

\newtheorem{definition}[theorem]{Definition}

\newcommand\supp{\mathrm{supp} \,}

\newcommand\Id{\mathrm{Id}}
\DeclareMathOperator*{\esssup}{ess\,sup}

\def\<{\langle}
\def\>{\rangle}
\def\d{{\rm d}}

\def\curl{{\rm curl \,}}

\def\E{\mathbb{E}}

\def\N{\mathbb{N}}
\def\P{\mathbb{P}}
\def\R{\mathbb{R}}
\def\T{\mathbb{T}}
\def\Z{\mathbb{Z}}
\def\O{\mathcal{O}}
\def\S{\mathcal{S}}

\begin{document}

\title{An elementary approach to mixing and dissipation enhancement by transport noise}
\author{Dejun Luo$^{1,2}$\,\footnote{Email: luodj@amss.ac.cn} \quad Bin Tang$^{3}$\,\footnote{Email: tangbin@math.pku.edu.cn} \quad Guohuan Zhao$^{1,2}$\,\footnote{Email: gzhao@amss.ac.cn} \bigskip \\
{\small $^1$Key Laboratory of RCSDS, Academy of Mathematics and Systems Science,}\\
{\small Chinese Academy of Sciences, Beijing 100190, China} \\
{\small $^2$School of Mathematical Sciences, University of Chinese Academy of Sciences,} \\
{\small Beijing 100049, China}\\
{\small $^3$LMAM, School of Mathematical Sciences, Peking University, Beijing 100871, China}}

\maketitle

\vspace{-15pt}

\begin{abstract}

We investigate the mixing properties of solutions to the stochastic transport equation $\d u= \circ \d W \cdot\nabla u$, where the driving noise $W(t,x)$ is white in time, colored and divergence-free in space. Furthermore, we prove the dissipation enhancement in the presence of a small viscous term. Applying our results, we also derive the mixing properties for a regularized stochastic 2D Euler equation.

\end{abstract}



\section{Introduction}\label{sec: introduction}

We consider on the torus $\T^d=\R^d/\Z^d\, (d\ge 2)$ the stochastic transport equation
  \begin{equation}\label{STE}
  \d u= \circ \d W \cdot\nabla u,
  \end{equation}
where $\circ\d$ means the Stratonovich stochastic differential, and $W=W(t,x)$ is a space-time noise, white in time, colored and divergence-free in space. Our noise takes the following form
  \begin{equation}\label{intro-noise}
  W(t,x)= \sqrt{C_d \kappa} \sum_{k\in \Z^d_0} \sum_{i=1}^{d-1} \theta_k \sigma_{k,i}(x) W^{k,i}_t,
  \end{equation}
where $C_d=d/(d-1)$, $\kappa>0$ is the noise intensity, and the coefficients $\{\theta_k\}_k\in \ell^2(\Z^d_0)$ are square summable and indexed by nonzero integer points $\Z^d_0$, $\{\sigma_{k,i}\}_{k,i}$ are divergence free vector fields on $\T^d$ and $\{W^{k,i}\}_{k,i}$ are standard independent complex-valued Brownian motions on some probability space $(\Omega, \mathcal F, \P)$, see Section \ref{subsec: noise} below for more details. Using these notations, equation \eqref{STE} can be written more precisely as
  \begin{equation*}
  \d u= \sqrt{C_d \kappa} \sum_{k,i} \theta_k \sigma_{k,i}\cdot\nabla u\circ \d W^{k,i} ,
  \end{equation*}
where $\sum_{k,i}$ stands for $\sum_{k\in \Z^d_0} \sum_{i=1}^{d-1}$. Assuming that the coefficients $\{\theta_k\}_k$ are radially symmetric in $k$ and $\|\theta \|_{\ell^2} =1$, then it is not difficult to show that the above equation has the equivalent It\^o form
  \begin{equation}\label{STE-Ito}
  \d u= \kappa\Delta u\,\d t+ \sqrt{C_d \kappa} \sum_{k,i} \theta_k \sigma_{k,i}\cdot\nabla u\,\d W^{k,i} .
  \end{equation}
We emphasize that the operator $\kappa\Delta$ is a ``fake'' dissipation term, since it is nullified by the martingale component in energy computations.

It is known (see \cite[Section 4.2]{Gal20}) that equation \eqref{STE-Ito} admits a pathwise unique solution $u$ for any $L^2(\T^d)$-initial datum $u_0$, which is assumed to have zero mean on $\T^d$. If the noise $W(t,x)$ is smooth in the spatial variable (e.g. if the coefficients $\{\theta_k \}_k$ decrease fast enough), then the $L^2$-norm of solutions is preserved $\P$-almost surely (abbreviated as $\P$-a.s.). However, as time evolves, the stochastic flow $\{\phi_t\}_{t\ge 0}$ generated by the random field $W(t,x)$ will gradually mix different parts of $\T^d$, heuristically, $u_t= u_0\circ \phi_t^{-1}$ takes both positive and negative values in any small region of $\T^d$. As a result, the solutions $u$ will vanish in the weak sense as $t\to \infty$. Our purpose is to rigorously prove the mixing properties of $u$, both in the average sense and in $\P$-a.s. sense. Moreover, we will prove mixing properties for the regularized stochastic 2D Euler equation by following \cite{BFM11} and using the Girsanov theorem, which relates such equation to \eqref{STE}. We will also establish the property of dissipation enhancement when there is a small dissipation term in \eqref{STE}.

The phenomena of mixing and dissipation enhancement by suitable fluid flows have been studied intensively by many authors, see for instance \cite{ACM19, Constantin, CZDE, FengIyer, GalGub, LTD11, Seis13} and the references therein.
In particular, explicit estimates on dissipation time scales in relation to the mixing rates of fluid flows can be found in \cite{CZDE, FengIyer}. These properties have important applications in fluid mechanics and engineering.
For example, it is shown in \cite{FFGT20} that if a sufficiently mixing advection term is added in the Cahn-Hilliard equation, then there will be no phase separation. Furthermore, strong advections can suppress the occurrence of singularities in models such as the Keller-Segel system \cite{IXZ21} and the Kuramoto-Sivashinsky equation \cite{CZDFM}, cf. \cite{FlGaLu21} for similar results in the stochastic setting. We would like to mention the advanced works \cite{BBPS21, BedroBlum21} by Bedrossion et al., which establish mixing and enhanced dissipation properties by solutions to stochastic Navier-Stokes equations, see \cite{DKK04} for some previous result and \cite{GessYar} for the case of Kraichnan noise.
These results are based on deep analysis of two-point motion of characteristics and a quantitative version of Harris' ergodic theorem. Recently, \cite{CZDG23} gave an explicit exponential mixing rate for self-similarly isotropic random fields.

In recent years, there are also lots of studies of stochastic partial differential equations (SPDEs) with Stratonovich transport noise, inspired by Galeati's work \cite{Gal20}. This work demonstrates that, through a rescaling of space covariance of the noises, the solution of \eqref{STE-Ito} converges to that of the deterministic heat equation.
In other words, an additional diffusion term appears in the limit, which can be interpreted as the eddy dissipation.
For related results on stochastic 2D Euler equations in the white noise regime, refer to \cite{FlLu20}. Results similar to those in \cite{Gal20} have been proved for various fluid dynamics models, see e.g. \cite{FGL} for 2D Euler equations and \cite{Luo21} for the 2D Boussinesq system. Moreover, the first named author of this article, along with collaborators, has improved these works in \cite{FGL21b} by establishing quantitative convergence estimates; 
see \cite{LuoWang23, LuoTang23} for related results on the dyadic model of turbulence and the Leray-$\alpha$ model. The method of \cite{FGL21b} is based on transforming stochastic fluid equations with Stratonovich transport noise into It\^o equations and then writing them in mild form; next, quantitative bounds between solutions of approximative and limiting equations are obtained by combining heat semigroup properties, estimates on the nonlinearities and the Gr\"onwall inequality.
However, a limitation of this method is its applicability only within finite time intervals, as shown in, for example, \cite[Theorem 1.1]{FGL21b}. Specifically, it remained unclear how to extend these estimates to capture mixing properties of solutions to \eqref{STE-Ito} over an infinite time horizon, as discussed following \cite[Theorem 1.7]{FGL21b}. One of the main purposes of the present work is to prove these properties using a completely different 
and relatively elementary approach;
see the discussions below Theorem \ref{thm: mixing-exp} for details.

Before stating our results, we introduce some notations. Given $\theta\in \ell^2(\Z^d_0)$ and $\beta \in \R$, we define the following norm-like quantity
  $$\|\theta \|_{h^\beta}^2 = \sum_k \theta_k^2 |k|^{2\beta}. $$
We will mainly use it for the cases $\beta=\pm 1$. We also introduce the constant
  \begin{equation} \label{def: C-theta-d}
  C(\theta,d)= \begin{cases}
  \pi^2 \|\theta \|_{h^{-1}}^2, & d=2; \\
  \frac{2\pi^2}5 \frac{d}{d-1} \|\theta \|_{h^{-1}}^2, & d\ge 3.
  \end{cases}
  \end{equation}
\subsection{Mixing by transport noise} \label{subsec: intro-mixing}

In this part, we first provide mixing properties of solutions to \eqref{STE} in the average sense.

\begin{theorem} \label{thm: mixing-exp}
Assume $u_0 \in L^2(\T^d)$ and $\theta \in \ell^2(\Z_0^d)$, then the solution $u$ to equation \eqref{STE-Ito} is exponentially mixing in the averaged sense:
\begin{itemize}
\item[\rm (i)] if $0<\beta \leq \frac{d}{4}$, then for any $\epsilon \in \big(0,\frac{\beta(d-2\beta)}{d^2} \big)$, there exists a constant $C_{\beta,\epsilon}$ such that
  $$\E \| u(t) \|_{H^{-\beta}}^2 \leq  C_{\beta,\epsilon} \, e^{-2 \kappa C(\theta,d) \, \big(\frac{\beta(d-2\beta)}{d^2}-\epsilon \big) \, t} \|u_0\|_{L^2}^{2};$$
\item[\rm (ii)] if $\beta > \frac{d}{4}$, there exists a constant $C$ that is only dependent on $\beta$ such that
  $$\E \| u(t) \|_{H^{-\beta}}^2 \leq C \, e^{- \kappa C(\theta,d) \, t/4} \|u_0\|_{L^2}^{2}.$$
\end{itemize}
\end{theorem}

Our main strategy for establishing the above results is to carefully analyze the energy spectrum of the solutions, allowing us to infer its exponential decay with time. Let $\{\hat u_k\}_k$ be the Fourier coefficients of $u$; then one can derive the following infinite dimensional ODE (see e.g. \cite[p. 859, Lemma 2]{Gal20} or Proposition \ref{prop: Fourier coefficient} below):
  \begin{equation*}
  \frac{\d}{\d t} \E |\hat{u}_{k}|^2 = 8 \pi^2 C_d \kappa \sum_{l} \theta_{l}^2 |\Pi_l^\perp k|^2 \big(\E |\hat{u}_{k-l}|^2 - \E |\hat{u}_{k}|^2\big), \quad k\in \Z^d_0,
  \end{equation*}
where $\Pi_l^\perp k=k - \frac{k\cdot l}{|l|^2} l$. For simplicity, let $Y_k= \E |\hat{u}_{k}|^2,\, k\in \Z^d_0$; in view of the definition of Sobolev norm $\| u \|_{H^{-\beta}}$ and H\"older's inequality, the estimates in Theorem \ref{thm: mixing-exp} will follow if we can prove the exponential decay of $t\mapsto \sum_{k}  Y_{k}^p(t)$ for suitable $p>1$. Indeed, under some regularity assumptions on $u_0$ and $\theta$, one can show that for any $p> 1$,
  \begin{equation*}
  \frac{\d}{\d t} \sum_{k}  Y_{k}^p = -4 \pi^2 C_{d} \kappa \, p  \sum_{l} \sum_{k} \theta_{l}^2 |\Pi_l^\perp k|^2 \big(Y_{k+l}-Y_{k} \big) \big(Y_{k+l}^{p-1}-Y_{k}^{p-1} \big)
  \end{equation*}
  (see Lemma \ref{lem: sum-Fourier-ODE}).
The structure of the right-hand side of the above identity reminds us of a Poincar\'e-type inequality:
  $$
    \sum_{k}  Y_{k}^p \lesssim \sum_{l} \sum_{k} \theta_{l}^2 |\Pi_l^\perp k|^2 \big(Y_{k+l}-Y_{k} \big) \big(Y_{k+l}^{p-1}-Y_{k}^{p-1} \big),
  $$
 and this suffices for our purpose. Unfortunately, we currently lack the means to directly prove such a result in the $d$-dimensional case. Instead, we decompose the quantity on the right-hand side of the above inequality into a countable sum of 1-dimensional series, and then use the following elementary inequality (see Lemma \ref{lem: sequence sum} below), which we believe to be of independent interest: let $p > 1$ and $\{a_{n}\}_{n \in \N}$ be a non-negative sequence in $\ell^{p}(\N)$, then
  $$\sum_{n \in \N} a_{n}^{p} \leq \frac{2p^2}{p-1} \sum_{n \in \N} (n+1)^2 (a_{n+1}^{p-1}-a_{n}^{p-1}) (a_{n+1}-a_{n}).$$
The readers can refer to Section \ref{subsec: mixing-average} for the details of proof. We would like to mention that deducing the mixing property by direct computations of $\E \| u(t) \|_{H^{-\beta}}^2$, as done in \cite{CZDG23}, may not work for general noises of the form \eqref{intro-noise}; see more discussions in Remark \ref{rem-direct-computation}.

For slightly more regular noises, we can also prove the $\P$-a.s. mixing properties of solutions to equation \eqref{STE-Ito}.

\begin{theorem} \label{thm: mixing-a.s.}
    If initial data $u_0 \in L^2(\T^d)$ and noise coefficient $\theta$ satisfies $\| \theta\|_{h^1}^2 < + \infty$, then for any $\lambda>0$ such that $\lambda/ \kappa$ is less than
    \begin{equation} \label{def: D}
       D(\theta,d) := \pi^2 \| \theta\|_{h^{-1}}^2 \times \begin{cases}
         \frac{1}{4}, & d=2; \\
         \frac{3}{20}, & d=3; \\
         \frac{4}{5} \frac{d-3}{d(d-1)}, & d\geq 4,
    \end{cases}
    \end{equation}
    and $0<q<\frac{D(\theta,d)  \, \kappa}{\lambda}-1$, there is a random constant $C_{\kappa,\theta,d}$ with finite $q$-moment such that
    $$\P \mbox{-a.s.}, \quad \|u(t) \|_{H^{-1}}^2 \leq C_{\kappa,\theta,d} \, e^{-\lambda t} \, \|u_0\|_{L^2}^2, \quad \forall \, t \in [0,+\infty) .$$
\end{theorem}

For the proof, we first show that, for a suitable $t_0>0$ and any $n\ge 1$, one has
  $$\E\Big[\sup_{nt_0\le t\le (n+1)t_0} \| u(t) \|_{H^{-1}}^2 \Big] \le 2\, \E\| u(nt_0) \|_{H^{-1}}^2; $$
combined with the estimates in Theorem \ref{thm: mixing-exp}, we get exponential decay rates for the quantity on the left-hand side. Then we can prove the $\P$-a.s. mixing property by following a standard procedure based on the Borel-Cantelli lemma, see Section \ref{subsec: mixing-P-a.s.}.

\subsection{Dissipation enhancement and Capi\'nski's conjecture}

In this part, we consider the following stochastic heat equation:
\begin{equation} \label{eq: stoch-heat}
    \d u =\lambda u \, \d t+  \nu \Delta u \, \d t + \sqrt{C_d \kappa} \sum_{k,i} \theta_{k} \sigma_{k,i} \cdot \nabla u \circ \d W_{t}^{k,i},
\end{equation}
where $\lambda \geq 0$ and $\nu >0$. The energy balance for \eqref{eq: stoch-heat} reads as
\begin{equation} \label{eq: energy-identity-heat}
   \P \mbox{-a.s.}, \quad \frac{\d}{\d t} \|u(t)\|_{L^2}^2 =  2 \lambda \|u(t)\|_{L^2}^2 -2 \nu \|\nabla u(t)\|_{L^2}^2.
\end{equation}
Then $\|u(t)\|_{L^2}^2 \leq \exp \{ (2 \lambda-8 \pi^2 \nu) t \} \|u_0\|_{L^2}^2$. When $\lambda=0$, Flandoli et al. used semigroup theory in \cite[Theorem 1.9]{FGL21b} to show that transport noises enhance the exponential decay rate of $\|u(t)\|_{L^2}^2$, for suitable choices of $\kappa$ and noise coefficient $\theta= \{\theta_k \}_k$.

For $\lambda > 4 \pi^2 \nu$, by simple energy estimate, the solution to deterministic heat equation (i.e. \eqref{eq: stoch-heat} with $\theta_k\equiv 0$) is instable in some low Fourier modes. As mentioned in the acknowledgements of \cite{FlaLuo21}, Capi\'nski conjectured around 1987 that stochastic transport in parabolic PDEs could stabilize such systems, similarly to the theory of stabilization by noise of Arnold et al. \cite{Arnold, ArnoldCW} in the finite dimensional case. Recently, Gess and Yaroslavtsev \cite{GessYar} studied the problem using the ergodicity of two-point motions. However, to apply their theory, one needs to verify the H\"ormander condition and the regularity condition on noise (see details in \cite[Section 3]{GessYar}); the regularity condition requires at least $ \| \theta \|_{h^1}^2<+\infty$.

We provide a different proof of the dissipation enhancement and Capi\'nski's conjecture, under the assumptions $\theta \in \ell^2(\Z_0^d)$ and $u_0 \in L^2(\T^d)$. Recall $D(\theta,d)$ defined in \eqref{def: D}.

\begin{theorem} \label{thm: heat}
Given initial datum $u_0 \in L^2(\T^d)$ and $\theta \in \ell^2(\Z_0^d)$, the solution $u$ to \eqref{eq: stoch-heat} satisfies
  \begin{equation*}
  \E \|u(t)\|_{L^2}^2 \leq C_0 \, \frac{8 \pi^2 \nu+ D(\theta,d)  \, \kappa}{2 \nu } \, e^{-(-2 \lambda + 8 \pi^2 \nu+D(\theta,d)  \, \kappa )  \, t} \|u_0\|_{L^2}^2,
  \end{equation*}
where $C_0 \geq \frac{1}{4 \pi^2}$. For any fixed $\lambda \geq 0$ and $\nu>0$, there exist appropriate noise parameters $(\kappa,\theta)$ such that $ 2 \lambda < 8 \pi^2 \nu+D(\theta,d)\, \kappa$, then the energy $\E \|u(t)\|_{L^2}^2$ decays exponentially.

Furthermore, if $ 2 \lambda < 8 \pi^2 \nu+D(\theta,d)  \, \kappa$, then for any
  $$\gamma\in \big((-2 \lambda + 8 \pi^2 \nu ) \vee 0, -2 \lambda + 8 \pi^2 \nu+D(\theta,d)\, \kappa\big),$$
there is a random constant $C= C(\omega)>0$ with $\E\, C^q <+\infty$ for any
  $$0<q < \min\Big\{\frac{-2 \lambda + 8 \pi^2 \nu+D(\theta,d)  \, \kappa}{\gamma}, \frac{D(\theta,d)  \, \kappa}{\gamma+2 \lambda - 8 \pi^2 \nu} \Big\}-1,$$
such that the solution $u$ satisfies $\P$-a.s.
  $$ \|u(t) \|_{L^2}^2 \leq C \, e^{-\gamma t} \, \|u_0\|_{L^2}^2,  \quad \forall \, t \in [0,+\infty).$$
\end{theorem}

\subsection{Mixing for regularized stochastic 2D Euler equations}

In this subsection, we fix some small $\alpha>0$ and consider the regularized stochastic 2D Euler equation on $\T^2$:
\begin{equation} \label{eq: 2D-Euler-approxiamate}
    \begin{cases}
        \d w^{(\alpha)} + u^{(\alpha)} \cdot \nabla w^{(\alpha)}\,\d t =  \sqrt{2 \kappa} \sum_{k} \theta_{k} \sigma_{k} \cdot \nabla w^{(\alpha)} \circ \d W_{t}^{k} , \\
        u^{(\alpha)} = {\curl}^{-1} (- \Delta/(4 \pi^2) )^{-\alpha/2} w^{(\alpha)},
    \end{cases}
\end{equation}
with initial condition $w^{(\alpha)}_0 \in L^2(\T^2)$, where ${\curl}^{-1}$ is the well-known Biot-Savart operator. Note that in the 2D case, for any $k\in \Z^2_0$, we write the divergence free vector field $\sigma_{k,1}$ simply as $\sigma_k$. If $\alpha=0$, the above system reduces to the stochastic 2D Euler equations with transport noise; considering the equation on the full space $\R^2$ and choosing Kraichnan noise with suitable parameter, Coghi and Maurelli \cite{CogMau23} proved the pathwise uniqueness for stochastic 2D Euler equation with general initial data in $L^p(\R^2)$. It seems that similar result holds for \eqref{eq: 2D-Euler-approxiamate} with $\alpha>0$, but we do not have a direct reference or a simple proof. Nevertheless, taking the noise coefficients
\begin{equation} \label{def: special-noise}
    \theta_k= \theta^{(\alpha)}_{k} = \frac{1}{K_{\alpha}} \frac{1}{|k|^{1+\alpha}} \  ( k \in \Z_0^2), \quad K_{\alpha}=\Big( \sum_{k} \frac{1}{|k|^{2+2\alpha}} \Big)^{1/2}
\end{equation}
and applying the Girsanov theorem, one can show the existence and uniqueness in law of weak solutions to \eqref{eq: 2D-Euler-approxiamate}, cf. Lemma \ref{lem: well-posedness 2D Euler-alpha} below or \cite[Remark 4.10]{GalLuo23} dealing with the logarithmically regularized nonlinearity which is weaker than here. Inspired by \cite{BBF, BFM11}, we will also use Girsanov's theorem to deduce mixing properties of solutions to \eqref{eq: 2D-Euler-approxiamate} from the same assertions of the stochastic transport equation \eqref{STE-Ito}.

\begin{theorem} \label{thm: mixing-Euler}
    Assume $\theta=\{ \theta^{(\alpha)}_{k} \}_k$ in \eqref{eq: 2D-Euler-approxiamate} and fix $R>0$. Given initial data $w_0^{(\alpha)} \in L^2(\T^2)$, let $w^{(\alpha)}$ be the unique energy controlled weak solution (see Definition \ref{def: solution-2D-Euler-approxiamte}) of \eqref{eq: 2D-Euler-approxiamate}. If $\|w^{(\alpha)}_0\|_{L^2}^2 \leq R$, then for any $\lambda>0$, there exist a noise intensity $\kappa(\lambda,R)$ and constant $C>0$ such that
    \begin{equation*}
        \E \|w^{(\alpha)}(t) \|_{H^{-1}}^2 \leq C \, e^{-\lambda t} \, \|w_0^{(\alpha)}\|_{L^2}^2, \quad \forall\, t \geq 0.
    \end{equation*}
\end{theorem}

We conclude this section with the structure of the paper. In Section \ref{sec: preliminary} we collect some preliminary results which will be used in the sequel. We first prove Theorem \ref{thm: mixing-exp} in Section \ref{subsec: mixing-average}, focusing on the 2D case for its intuitiveness. The more intricate high-dimensional case is deferred to the appendix \ref{appendix: proof}. We will prove Theorem \ref{thm: mixing-a.s.} in Section \ref{subsec: mixing-P-a.s.} and also provide in Section \ref{subsec: positive-norm} some exponential upper bounds on the growth rate of positive Sobolev norms. Sections \ref{sec: heat-equation} and \ref{sec: 2D Euler} are devoted to proving Theorems \ref{thm: heat} and \ref{thm: mixing-Euler}, respectively.

\section{Preliminaries} \label{sec: preliminary}

This section contains some preparatory work. We introduce in Section \ref{subsec: setting} some notations and conventions used throughout the paper, while Section \ref{subsec: noise} is devoted to the precise choice of noise $W(t,x)$ and some previous results on it. In Section \ref{subsec: tool}, we define the solution $u$ of stochastic transport equation \eqref{STE-Ito}, introduce some existing results for solution $u$, and provide a key inequality used in the proof of Theorem \ref{thm: Fourier-coefficient-exp-decay}.

\subsection{Notions and functional setting} \label{subsec: setting}

Let $\T^d=\R^d/\Z^d$ be the $d$-dimensional torus and $\Z_0^d = \Z^d \setminus \{0\}$ be the set of nonzero lattice points. As we assume the noise is spatially divergence-free (see Section \ref{subsec: noise}), the spatial average of all solutions considered in this paper are preserved. Hence, we shall assume for simplicity that the function spaces in this paper consist of functions on $\T^d$ with zero average.

Let $e_k(x):=\exp \{ 2 \pi {\rm i} k \cdot x\}$, for any $k \in \Z_0^d$ and $x \in \T^d$. Then $\{e_k \}_k$ is the usual complex basis of $L^2(\T^d; \mathbb C)$. Let $\hat{u}_{k}=\<u,e_{k}\>$ be the Fourier coefficient of $u$. For any $\beta \in \R$, we define Sobolev space $H^{\beta}(\T^d)$ as
\begin{equation*}
    H^{\beta}(\T^d)=\Big\{ u = \sum_{k \in \Z_0^d} \hat{u}_{k} \, e_k ; \quad  \| u \|_{H^\beta}^2:= \sum_{k \in \Z_0^d} |2 \pi k|^{2\beta} |\hat{u}_k|^2 < +\infty \quad \mbox{and} \quad \overline{\hat{u}_k} =\hat{u}_{-k} \Big\}.
\end{equation*}
Following standard notation, we denote the space $H^{0}(\T^d)$ by $L^2(\T^d)$ with the same norm $\|u\|_{L^2}=\|u\|_{H^0}$ and inner product $\langle u, v\rangle=\sum_{k \in \Z_0^d}\left(\hat{u}_{k} \cdot \hat{v}_{-k}\right)$. Let $C([0,T]; L^2_w(\T^d))$ be the space of functions $f: [0,T] \rightarrow L^2(\T^d)$ which are continuous with respect to the weak topology of $L^2(\T^d)$.

For $1 \leq p < +\infty $, we define $\ell^p (\Z_0^d)$ as the space of $p$-order summable sequences indexed by $\Z_0^d$ with the norm $ \|a\|_{\ell^p}: = \big(\sum_{k \in \Z_0^d} |a_k|^p \big)^{1/p}$ and define $\ell^{\infty} (\Z_0^d)$ as the space of bounded sequences with the norm $ \|a\|_{\ell^{\infty}}: = \sup_{k \in \Z_0^d} |a_k| $. For $\beta \in \R$, we define the norm $\| a \|_{h^{\beta}}^2 = \sum_{k \in \Z_0^d} |k|^{2 \beta} |a_k|^2  $ and the space $h^{\beta}(\Z_0^d)$ as all sequences $\{a_k \}_{k \in \Z_0^d}$ satisfying $\| a\|_{h^{\beta}}^2 < +\infty$. Let $\N:=\{0,1,2, \cdots \}$ be the set of non-negative integers. Similar as $\ell^p (\Z_0^d)$, we define $\ell^{p}(\N)$ with the norm $ \|a\|_{\ell^p}: = \big(\sum_{k \in \N} |a_k|^p \big)^{1/p}$.

For any $k,l \in \Z_0^d$, we define the projection operator $\Pi_l k= \frac{k \cdot l}{|l|^2} l $ and its complement operator $\Pi^{\perp}_l k= k-\frac{k \cdot l}{|l|^2} l$. For any two linearly independent vectors $l_1, l_2 \in \mathbb{Z}_0^d$, we denote the plane spanned by them as
\begin{equation*}
    P(l_1,l_2)=\{l \in \R^d; \quad l=a \, l_1+ b \, l_2, \quad a,b \in \R  \},
\end{equation*}
and define $\Pi_{(l_1,l_2)} m$ as the orthogonal projection of $m \in \mathbb{Z}_0^d$ on the plane $P(l_1, l_2)$. Let $\Id$ be the identity operator, we define $\Pi_{(l_1,l_2)}^{\perp} = \Id - \Pi_{(l_1,l_2)} $. Let $P_Z^{\perp}(l_1,l_2)=\{\Pi_{(l_1,l_2)}^{\perp} m; \, m \in \Z_0^d \}$, then for any $h \in P_Z^{\perp}(l_1,l_2)$, we define the set of integer points based on $h$ as
\begin{equation*}
    \S_h(l_1,l_2)=\{l \in \Z_0^d; \quad l=a \, l_1+ b \, l_2+h, \quad a,b \in \R  \}.
\end{equation*}

\subsection{Choice of noise} \label{subsec: noise}

As mentioned in Section \ref{sec: introduction}, the space-time noise takes the form
\begin{equation} \label{def: noise}
  W(t, x)= \sqrt{C_{d} \kappa}\, \sum_{k \in \Z_0^d} \sum_{i=1}^{d-1} \theta_{k} \sigma_{k, i}(x) W^{k, i}_{t} ,
\end{equation}
where $C_{d}=d/(d-1)$ is a normalizing constant, $\kappa>0$ is the noise intensity, and $\theta \in \ell^{2}(\Z^d_0)$. $\{W^{k, i}:k\in\mathbb{Z}^{d}_{0},  i=1, \ldots, d-1\}$ are standard complex Brownian motions defined on a filtered probability space $(\Omega,  \mathcal F,  \mathcal F_t,  \P)$  satisfying
\begin{equation} \label{noise.1}
    \overline{W^{k, i}} = W^{-k, i},  \quad\big[W^{k, i}, W^{l, j} \big]_{t}= 2t \delta_{k, -l} \delta_{i, j} .
\end{equation}
$\{\sigma_{k, i}: k\in\mathbb{Z}^{d}_{0},  i=1, \ldots, d-1\}$ are divergence-free vector fields on $\T^d$ defined as
\begin{equation*}
    \sigma_{k, i}(x) = a_{k, i} e_{k}(x) ,
\end{equation*}
where $\{a_{k, i}\}_{k, i}$ is a subset of the unit sphere $\mathbb{S}^{d-1}$ such that: i) $a_{k, i}=a_{-k, i}$ for all $k\in \mathbb{Z}^{d}_{0}, \,  i=1, \ldots, d-1$; ii) for fixed $k$, $\{a_{k, i}\}_{i=1}^{d-1}$ is an orthonormal basis of $k^{\perp}=\{y\in\mathbb{R}^{d}:y\cdot k=0 \}$. It holds that $\sigma_{k, i}\cdot \nabla e_k = \sigma_{k, i}\cdot \nabla e_{-k} \equiv 0$ for all $k\in \Z^d_0$ and $1\leq i\leq d-1$. For simplicity, we shall write $\sum_{k \in \Z_0^d}$ and $\sum_{k \in \Z_0^d} \sum_{i=1}^{d-1}$ as $\sum_k$ and $\sum_{k,i}$ in the sequel, respectively. When $d=2$, we will simply write $\sigma_{k,1}$ and $a_{k,1}$ as $\sigma_k$ and $a_k$, respectively.

We shall always assume that $\theta $ is symmetric:
\begin{equation}\label{def: theta_sym}
    \theta_{k}=\theta_{l}, \quad \forall k, l \in \Z_{0}^d \quad \mbox{satisfying} \quad |k|=|l|,
\end{equation}
and $\|\theta \|_{\ell^2} =1$. We define the support of $\theta$ as $\supp \theta := \{ k \in \Z_0^d ; \, \theta_k \neq 0 \}$.

The following useful lemma can be applied to derive the It\^o formulation \eqref{STE-Ito} from the Stratonovich equation \eqref{STE}. In the sequel, it will play roles in the proofs of Lemmas \ref{lem: sum-Fourier-ODE} and \ref{lem: uniform-bounded}, Theorem \ref{thm: positive-sobolev} and so on. The proof of Lemma \ref{lem: matrix-property} can be found in \cite[Section 5.1]{FangLuo18} or \cite[Section 2.3]{Gal20}.

\begin{lemma} \label{lem: matrix-property}
    Let $\{a_{k, i}\}_{k,i}$ be defined as above and $\theta \in \ell^2(\T^d)$ satisfy the symmetry condition \eqref{def: theta_sym}, then we have the following identity
    $$C_{d} \sum_{k,i} \theta_{k}^2 \, a_{k,i} \otimes a_{k,i}= C_{d} \sum_{k} \theta_{k}^2 \left( \Id - \frac{k}{|k|} \otimes \frac{k}{|k|} \right) = \|\theta \|_{\ell^2}^2 \,\Id.$$
\end{lemma}

\subsection{Basic results for \eqref{STE-Ito} and a discrete Poincar\'e type inequality} \label{subsec: tool}

Firstly, we define the solution to stochastic transport equation \eqref{STE-Ito}, for which the existence and pathwise uniqueness are proved in \cite[Theorems 2 and 3]{Gal20}. Our assumptions in this paper are either identical or stronger than those in \cite{Gal20}, ensuring that the solution $u$ to equation \eqref{STE-Ito} is always existent and unique.

\begin{definition} \label{def: solution-transport-equation}
    Let $(\Omega, \mathcal F, \mathcal F_t, \P)$ be a filtered probability space, and the space-time noise $W(t,x)$ defined in this probability space has the form \eqref{def: noise}. We say that an $\mathcal F_t$-progressively measurable, $L^2(\T^d)$-valued process $u$, with paths in $C([0,\infty); L^2_w(\T^d))$ is a solution of the equation \eqref{STE-Ito}, if for every $\phi \in C^{\infty}(\T^d)$, $\P$-a.s. the following identity holds for all $t \geq 0$:
    \begin{equation}\label{def-solu-STE}
        \< u(t), \phi \> - \< u(0), \phi \> = \kappa \int_{0}^{t} \< u(s), \, \Delta \phi \> \, \d s - \sqrt{C_d \kappa} \sum_{l,i} \theta_l \int_{0}^{t} \<  u(s), \sigma_{-l,i} \cdot \nabla \phi \> \, \d W_{s}^{l,i},
    \end{equation}
    and $\| u(t)\|_{L^2}^2 \leq \| u(0)\|_{L^2}^2$ holds $\P$-a.s. for all $t \ge0$.
\end{definition}

As mentioned in Section \ref{subsec: intro-mixing}, we want to prove mixing properties of the solution $u$ by studying its energy spectrum. To this end, we will need the equations for $|\hat{u}_{k}(t)|^2$ and $\E |\hat{u}_{k}(t)|^2$ which have been derived in \cite[Lemma 2]{Gal20}. Equation \eqref{eq: Fourier ODE-1} will be used in the proof of mixing in the average sense, and equation \eqref{eq: Fourier-coefficient-norm} will be used in the proof of mixing in $\P$-a.s. sense. Given the significance of the above results and the completeness of this article, we include the complete proofs for these results.

\begin{proposition} \label{prop: Fourier coefficient}
    Let $\hat{u}_{k}=\<u,e_{k}\>,\, k\in \Z^2_0$ be the Fourier coefficients of the solution $u(t)$ to equation \eqref{STE-Ito}; then, $|\hat{u}_{k}|^2$ satisfies
    \begin{align}
        \d |\hat{u}_{k}|^2 &= - 8 \pi^2 \kappa |k|^2 |\hat{u}_{k}|^2 \, \d t + 8 \pi^2 C_d \kappa \sum_{l,i} \theta_{l}^2 (a_{l,i} \cdot k)^2 |\hat{u}_{k-l}|^2 \, \d t \nonumber  \\
        & \quad + 2 \pi {\rm i} \sqrt{C_d \kappa} \sum_{l,i} \theta_l (a_{l,i} \cdot k) ( \overline{\hat{u}_k} \hat{u}_{k-l}-\hat{u}_k \overline{\hat{u}_{k+l}} ) \, \d W_{t}^{l,i}. \label{eq: Fourier-coefficient-norm}
    \end{align}
    Consequently, $\E |\hat{u}_{k}|^2$ satisfies the infinite-dimensional ODE:
    \begin{equation} \label{eq: Fourier ODE-1}
    \frac{\d}{\d t} \E |\hat{u}_{k}|^2 = - 8 \pi^2 \kappa |k|^2 \E |\hat{u}_{k}|^2 + 8 \pi^2 C_d \kappa \sum_{l,i} \theta_{l}^2 (a_{l,i} \cdot k)^2 \E |\hat{u}_{k-l}|^2 .
    \end{equation}
\end{proposition}

\begin{proof}
    By \eqref{def-solu-STE}, the Fourier coefficient $\hat{u}_k$ satisfies:
    $$ \d \hat{u}_{k} = \d \< u, \, e_k \> = - 4 \pi^2 \kappa |k|^2 \hat{u}_{k} \, \d t + \sqrt{C_d \kappa} \sum_{l,i} \theta_l \< \sigma_{l,i} \cdot \nabla u, \, e_k \> \,  \d W_{t}^{l,i} .$$
    Due to the vector fields $\{\sigma_{l,i} \}_{l,i}$ are divergence-free, we obtain
    \begin{align}
    \d \hat{u}_{k} = \d \< u, \, e_k \>  & = - 4 \pi^2 \kappa |k|^2 \hat{u}_{k} \, \d t - \sqrt{C_d \kappa} \sum_{l,i} \theta_l \< u, \, \sigma_{-l,i} \cdot \nabla e_k \> \, \d W_{t}^{l,i} \nonumber \\
    & = - 4 \pi^2 \kappa |k|^2 \hat{u}_{k} \, \d t + 2 \pi {\rm i} \sqrt{C_d \kappa} \sum_{l,i} \theta_l (a_{l,i} \cdot k) \hat{u}_{k-l} \, \d W_{t}^{l,i}. \label{eq: Fourier coefficient equation}
    \end{align}
    Using the identity $|\hat{u}_{k}|^{2}=\hat{u}_{k}\overline{\hat{u}_k}$ and It\^o formula, we have
    $$ \d |\hat{u}_{k}|^2 =\hat{u}_{k} \, \d \overline{\hat{u}_k} +\frac{1}{2} \, \d \big[ \hat{u}_{k} , \, \overline{\hat{u}_k} \,\big] + \overline{\hat{u}_k} \, \d \hat{u}_{k} + \frac{1}{2} \, \d \big[ \, \overline{\hat{u}_k}, \, \hat{u}_{k} \big] .$$
    By equation \eqref{eq: Fourier coefficient equation}, we get
    \begin{align*}
        \d \big[ \hat{u}_{k} , \,  \overline{\hat{u}_k} \, \big] & = 4 \pi^2 C_d \kappa \, \d \Bigl[\, \sum_{l,i} \theta_l (a_{l,i} \cdot k) \hat{u}_{k-l} W_{t}^{l,i}, \, \sum_{m,j} \overline{ \theta_m (a_{m,j} \cdot k) \hat{u}_{k-m} W_{t}^{m,j}}  \, \Bigr] \\
        & =4 \pi^2 C_d \kappa \sum_{l,i} \sum_{m,j} \theta_{l} \theta_{m} (a_{l,i} \cdot k) (a_{m,j} \cdot k) \hat{u}_{k-l}  \overline{\hat{u}_{k-m}} \, \d \bigl[ W^{l,i} , W^{-m,j} \bigr]_t \\
        & = 8 \pi^2 C_d \kappa \sum_{l,i} \theta_{l}^2 (a_{l,i} \cdot k)^2 |\hat{u}_{k-l}|^2 \, \d t.
    \end{align*}
    Therefore, $\d \big[ \hat{u}_{k} , \,  \overline{\hat{u}_k} \, \big]= \d \big[ \, \overline{\hat{u}_k}, \, \hat{u}_{k} \big] $ and it holds
    \begin{align*}
    \d |\hat{u}_{k}|^2 &= \hat{u}_{k} \, \d \overline{\hat{u}_k}+ \overline{\hat{u}_k} \, \d \hat{u}_{k} + 8 \pi^2 C_d \kappa \sum_{l,i} \theta_{l}^2 (a_{l,i} \cdot k)^2 |\hat{u}_{k-l}|^2 \, \d t  \\
    &= - 8 \pi^2 \kappa |k|^2 |\hat{u}_{k}|^2 \, \d t + 8 \pi^2 C_d \kappa \sum_{l,i} \theta_{l}^2 (a_{l,i} \cdot k)^2 |\hat{u}_{k-l}|^2 \, \d t \nonumber  \\
    & \quad + 2 \pi {\rm i} \sqrt{C_d \kappa} \sum_{l,i} \theta_l (a_{l,i} \cdot k) ( \overline{\hat{u}_k} \hat{u}_{k-l}-\hat{u}_k \overline{\hat{u}_{k+l}} ) \, \d W_{t}^{l,i}.
    \end{align*}
    So equation \eqref{eq: Fourier-coefficient-norm} holds, and taking expectation of both sides of \eqref{eq: Fourier-coefficient-norm} gives us \eqref{eq: Fourier ODE-1}.
\end{proof}

If the initial data $u_0$ is smooth and coefficients $\{\theta_k\}_k$ have compact support, then we are in the good case where the solution of equation \eqref{STE-Ito} is sufficiently regular. This regularity can be deduced from the nice properties of the stochastic flow corresponding to equation \eqref{STE-Ito}, see Lemma \ref{lem: flow} below. It will allow us to exchange the order of summation in the proofs of Lemmas \ref{lem: sum-Fourier-ODE} and \ref{lem: uniform-bounded}, Theorem \ref{thm: positive-sobolev} and so on.

\begin{lemma} \label{lem: flow}
    If $\supp \theta \subset \Z_0^d$ is compact and the initial data $u_0$ is smooth, then the solution to equation \eqref{STE-Ito} is $u(t,x)=u_0(\phi^{-1}_t(x))$, where $\{\phi_t\}_{t \geq 0}$ is the stochastic flow of $C^{\infty}$-diffeomorphisms generated by
    \begin{equation} \label{eq: flow}
       \d \phi_t(x) = -\sqrt{C_d \kappa} \sum_{k,i} \theta_k \sigma_{k,i}(\phi_t (x)) \circ \d W_{t}^{k,i}, \quad \phi_0(x)=x.
    \end{equation}
    Since $\{\sigma_{k,i}\}_{k,i}$ are divergence-free, $\{\phi_t\}_{t \geq 0}$ preserves the volume measure on $\T^d$. Moreover, for any $T>0$, $p \geq 2$ and non-negative multiple indices $\alpha=(\alpha_1,\alpha_2,\cdots, \alpha_d)\in \N^d$ with norm $|\alpha|=\sum_{i=1}^d \alpha_i$, we have
    \begin{equation*}
        \sup_{x \in \T^d} \E \sup_{t \in [0,T]} \bigg|\frac{\partial^{|\alpha|}}{\partial x_1^{\alpha_1} \, \partial x_2^{\alpha_2} \, \cdots  \partial x_d^{\alpha_d} } \phi_t^{-1}(x) \bigg|^p < + \infty.
    \end{equation*}
    Thus, for any $\beta>0$, the solution $u \in L^2(\Omega; L^{\infty}([0,T]; H^{\beta}(\T^d) )) \cap L^{\infty}(\Omega; L^{\infty}([0,T]; L^{2}(\T^d)) $.
\end{lemma}
\begin{proof}
    The regularity properties of the stochastic flow \eqref{eq: flow} have been proved in \cite{Kun} and \cite[Lemma 3.9]{GessYar}. We only prove the regularity of solution $u$ using these properties. The volume-preserving property of $\phi_t$ implies $\|u(t)\|_{L^2} = \|u_0\|_{L^2}$ $\P$-a.s., so $u \in L^{\infty}(\Omega; L^{\infty}([0,T]; L^{2}(\T^d))$. For $\beta=1$,
    \begin{align*}
        \E \sup_{t \in [0,T]} \|u(t)\|_{H^1}^2 &= \E \sup_{t \in [0,T]} \sum_{i=1}^d \Big\|\frac{\partial}{\partial x_i} u_0(\phi_t^{-1}(x)) \Big\|_{L^2}^2 \\
        & \leq \E \sup_{t \in [0,T]} \sum_{i=1}^d \big\| (\nabla u_0)\circ \phi_t^{-1} \big\|_{L^\infty}^2 \Big\| \frac{\partial}{\partial x_i} \phi_t^{-1} \Big\|_{L^2}^2 \\
        & \le \| \nabla u_0 \|_{L^\infty}^2 \sum_{i=1}^d \E \sup_{t \in [0,T]} \Big\| \frac{\partial}{\partial x_i} \phi_t^{-1} \Big\|_{L^2}^2.
    \end{align*}
    Using Fubini's theorem and the regularity estimate of stochastic flow \eqref{eq: flow}, we obtain
    \begin{align*}
        \E \sup_{t \in [0,T]} \|u(t)\|_{H^1}^2 & \leq \| \nabla u_0 \|_{L^\infty}^2 \sum_{i=1}^d  \sup_{x \in \T^d} \E \sup_{t \in [0,T]} \Big| \frac{\partial}{\partial x_i} \phi_t^{-1} (x) \Big|^2 < +\infty.
    \end{align*}
    For other $\beta>0$, similarly we can get the solution $u$ belongs to $ L^2(\Omega; L^{\infty}([0,T]; H^{\beta}(\T^d) ))$.
\end{proof}

In this paper, we often prove our results first in the case that the initial data $u_0$ is smooth and coefficients $\{\theta_k\}_k$ have compact support,
and then extend them to the general case by approximation and the following classical Lemma (see \cite[Theorem 3.7]{Bre}).

\begin{lemma} \label{lem: convegence}
    Let $E$ be a Banach space and $C$ be a convex subset of $E$. Then $C$ is closed in the weak topology of $E$ if and only if it is closed in the strong topology of $E$.
\end{lemma}

The following Poincar\'e type inequality in $\N$ is necessary to prove Theorem \ref{thm: mixing-exp}.

\begin{lemma} \label{lem: sequence sum}
Let $p > 1$ and $\{a_{n}\}_{n \in \N}$ be a non-negative sequence in $\ell^{p}(\N)$, then
$$\sum_{n \in \N} a_{n}^{p} \leq \frac{2p^2}{p-1} \sum_{n \in \N} (n+1)^2 (a_{n+1}^{p-1}-a_{n}^{p-1}) (a_{n+1}-a_{n}).$$
In particular, when $p=2$, the sequence $\{a_{n}\}_{n \in \N}$ need not be non-negative.
\end{lemma}
\begin{proof}
For any sequence $\{c_{n}\}_{n \in \N}$ with only finite nonzero terms, we have
    \begin{align*}
         \sum_{n \in \N} a_{n} c_{n} & = \sum_{n \in \N}  (n+1) a_{n} c_{n} - \sum_{n \in \N} n  a_n c_n = \sum_{n \in \N}  (n+1) a_{n} c_{n} - \sum_{n \in \N} (n+1)  a_{n+1} c_{n+1}  \\
         & = - \sum_{n \in \N} (n+1)\big( a_{n+1}c_{n+1}-a_{n}c_{n} \big) .
    \end{align*}
    We set $c_{n}^N = a_{n}^{p-1}$ if $0 \leq n \leq N$ and $c_{n}^N = 0$ if $n \geq N+1$. Substituting $\{c_n^N\}_{n \in \N}$ into the above equation, we obtain
    \begin{equation} \label{eq: lem-sequence-sum}
        \sum_{n=0}^{N} a_{n}^p = - \sum_{n=0}^{N-1} (n+1) (a_{n+1}^{p}-a_{n}^{p}) + (N+1)a_{N}^{p}.
    \end{equation}

When $p>1$, the functions $x^p$ and $x^{\frac{p}{p-1}}$ are convex, then for any $a,b \geq 0$, it holds
    \begin{align*}
        p\, b^{p-1} \, (a-b) & \leq  a^{p}-b^p \leq p \, a^{p-1} \, (a-b) ;\\
        \frac{p}{p-1} \, b \, (a^{p-1}-b^{p-1}) & \leq a^{p}-b^p \leq \frac{p}{p-1} \, a \, (a^{p-1}-b^{p-1}).
    \end{align*}
Combining the above two inequalities, we have
    $$|a^{p}-b^{p}| \leq \frac{p}{\sqrt{p-1}} \sqrt{a^p+b^p} \, \sqrt{(a^{p-1}-b^{p-1})(a-b)}.$$
When $p=2$, the formula also holds without requiring $a,b \geq 0$. Using triangle inequality and H\"older's inequality, we obtain
    \begin{align*}
        \left| \sum_{n=0}^{N-1} (n+1) (a_{n+1}^p-a_{n}^p) \right| & \leq \frac{p}{\sqrt{p-1}} \Big( \sum_{n=0}^{N-1} (a_{n}^{p}+a_{n+1}^{p}) \Big)^{1/2} \\
        & \quad \times \Big( \sum_{n=0}^{N-1} (n+1)^2 (a_{n+1}^{p-1}-a_{n}^{p-1}) (a_{n+1}-a_{n}) \Big)^{1/2}.
    \end{align*}
From equality \eqref{eq: lem-sequence-sum} and the above estimate, we have
    \begin{equation*}
        \sum_{n=0}^{N} a_{n}^p \leq \sqrt{\frac{2 p^2}{p-1}} \Big( \sum_{n=0}^{N} a_{n}^p\Big)^{1/2} \Big( \sum_{n=0}^{N-1} (n+1)^2 (a_{n+1}^{p-1}-a_{n}^{p-1}) (a_{n+1}-a_{n}) \Big)^{1/2}+(N+1)a_{N}^{p}.
    \end{equation*}

We claim that for any $\epsilon>0$, there exists a subsequence $\{N_j\}_{j \in \mathbb{N}}$ such that $(N_j+1) a_{N_j}^p \leq \epsilon$. If this claim does not hold, then there exists $L>0$, such that for all $n>L$, we have $(n+1) a_n^p > \epsilon$ for all $n>L$. This results in $a_n^p > \frac{\epsilon}{n+1}$ for all $n>L$, which contradicts with the fact that $\{a_n\} \in \ell^p(\N)$. Thus, the claim is true, and we can conclude with
    \begin{equation*}
        \sum_{n=0}^{N_j} a_{n}^p \leq \sqrt{\frac{2 p^2}{p-1}} \Big( \sum_{n=0}^{N_j} a_{n}^p\Big)^{1/2} \Big( \sum_{n=0}^{N_j-1} (n+1)^2 (a_{n+1}^{p-1}-a_{n}^{p-1}) (a_{n+1}-a_{n}) \Big)^{1/2}+\epsilon.
    \end{equation*}
Taking $j\rightarrow +\infty$ and $\epsilon\rightarrow 0$, we get
    $$\sum_{n \in \N} a_{n}^p \leq \sqrt{\frac{2 p^2}{p-1}} \Big( \sum_{n \in \N} a_{n}^p\Big)^{1/2} \Big( \sum_{n\in N} (n+1)^2 (a_{n+1}^{p-1}-a_{n}^{p-1}) (a_{n+1}-a_{n}) \Big)^{1/2}.$$
Upon simplification, we obtain the desired result.
\end{proof}

\section{Mixing for stochastic transport equation} \label{sec: transport equation}

The purpose of this section is to prove Theorems \ref{thm: mixing-exp} and \ref{thm: mixing-a.s.}. We also establish estimates of positive Sobolev norms for the solution $u$ to stochastic transport equation \eqref{STE-Ito} in Section \ref{subsec: positive-norm}, which provides us with lower bounds of the mixing rate.

\subsection{Mixing in the average sense} \label{subsec: mixing-average}

In this subsection, we will deal with the mixing of the solution $u$ to equation \eqref{STE-Ito} in the average sense. Let $Y_k(t)=\E|\hat{u}_k(t)|^2$ be the second moment of the Fourier coefficient $\hat{u}_k(t)$, $k\in \Z^d_0$. Firstly, we prove Lemma \ref{lem: sum-Fourier-ODE}, which implies that the $\ell^{p}$ norm of sequence $\{ Y_k \}_{k}$ is decreasing for any $p>1$. Next, applying Lemma \ref{lem: sequence sum}, we show in Theorem \ref{thm: Fourier-coefficient-exp-decay} that the $\ell^{p}$ norm of $\{Y_k\}_k$ decays exponentially fast. Finally, we use Theorem \ref{thm: Fourier-coefficient-exp-decay} and H\"older inequality to get the mixing properties of solutions to \eqref{STE-Ito} in the average sense (see Theorem \ref{thm: mixing-exp}).

\begin{lemma} \label{lem: sum-Fourier-ODE}
    If $\supp \theta \subset \Z_0^d$ is compact and $u_0$ is smooth, then for any $p\geq 1$ we have
    \begin{equation*}
        \frac{\d}{\d t} \sum_{k}  Y_{k}^p = -4 \pi^2 C_{d} \kappa \, p  \sum_{l} \sum_{k} \theta_{l}^2 |\Pi_l^\perp k|^2 \big(Y_{k+l}-Y_{k} \big) \big(Y_{k+l}^{p-1}-Y_{k}^{p-1} \big).
    \end{equation*}
\end{lemma}

\begin{proof}
    For any $p \geq 1$, equation \eqref{eq: Fourier ODE-1} yields the following integral equation:
    \begin{align}
    Y_{k}^p(t)-Y_{k}^p(0) & = 8 \pi^2 \kappa \, p \int_{0}^{t} \Big(-  |k|^2 Y_{k}^p (s) + C_d \sum_{l,i} \theta_{l}^2 (a_{l,i} \cdot k)^2 Y_{k-l}(s) Y_{k}^{p-1}(s) \Big) \, \d s \nonumber \\
    & =  8 \pi^2 C_d \kappa \, p \sum_{l,i} \theta_{l}^2 (a_{l,i} \cdot k)^2 \int_{0}^{t} \big(Y_{k-l}(s)-Y_{k}(s)  \big)  Y_{k}^{p-1}(s) \, \d s, \label{eq: Fourier-coefficient-norm-2}
    \end{align}
    where the last step is due to Lemma \ref{lem: matrix-property}. By the fact $\sum_{k} Y_{k}(t) = \E \|u(t)\|_{L^2}^2 \leq \|u(0)\|_{L^2}^2 < +\infty$, we obtain $\sup\limits_{k} Y_{k}(t) \leq \|u(0)\|_{L^2}^2 < + \infty$ and
    $$\sum_{k} Y_{k}^{p}(t) \leq \sup\limits_{k} Y_{k}^{p-1}(t) \sum_{k} Y_{k}(t) \leq \|u(0)\|_{L^2}^{2p}  < +\infty, \quad \forall \, t \geq 0.$$
    Since $\sum_{k} Y_{k}^p (t)<+\infty$, summing equation \eqref{eq: Fourier-coefficient-norm-2} over $k \in \Z_0^d$ yields
    \begin{equation} \label{eq: Fourier sum}
        \sum_{k} Y_{k}^p(t) - \sum_{k} Y_{k}^p (0)= 8 \pi^2 C_d \kappa \, p \sum_{k} \sum_{l,i} \theta_{l}^2 (a_{l,i} \cdot k)^2 \int_{0}^{t} \big(Y_{k-l}(s)-Y_{k}(s) \big) Y_{k}^{p-1}(s) \, \d s.
    \end{equation}

According to Lemma \ref{lem: flow}, $u$ belongs to $ L^2(\Omega; L^{\infty}([0,T]; H^1(\T^d) ))$. Therefore, the series on the right-hand side of equation \eqref{eq: Fourier sum} converges absolutely. By Fubini's theorem, we have
    \begin{equation} \label{eq: Fourier sum-2}
        \sum_{k} Y_{k}^p(t) - \sum_{k} Y_{k}^p (0)= 8 \pi^2 C_d \kappa \, p \sum_{l,i} \sum_{k} \theta_{l}^2 (a_{l,i} \cdot k)^2 \int_{0}^{t} \big(Y_{k-l}(s)-Y_{k}(s) \big) Y_{k}^{p-1}(s) \, \d s.
    \end{equation}
    For any fixed $l \in \Z_0^d$, since $a_{l,i} \cdot l=0$, substituting $k$ for $\Tilde{k}+l$ on the right-hand side of equation \eqref{eq: Fourier sum-2} yields
    \begin{equation*}
        \sum_{k} Y_{k}^p(t) - \sum_{k} Y_{k}^p (0)= 8 \pi^2 C_d \kappa \, p \sum_{l,i} \sum_{\Tilde{k}} \theta_{l}^2 (a_{l,i} \cdot \Tilde{k})^2 \int_{0}^{t} \big(Y_{\Tilde{k}}(s)-Y_{\Tilde{k}+l}(s) \big) Y_{\Tilde{k}+l}^{p-1}(s) \, \d s.
    \end{equation*}
    Substituting $l$ for $-\Tilde{l}$ in the equation \eqref{eq: Fourier sum-2} and using the symmetry \eqref{def: theta_sym} of $\theta$ yield
    \begin{equation*}
        \sum_{k} Y_{k}^p(t) - \sum_{k} Y_{k}^p (0)= 8 \pi^2 C_d \kappa \, p \sum_{\Tilde{l},i} \sum_{k} \theta_{\Tilde{l}}^2 (a_{\Tilde{l},i} \cdot k)^2 \int_{0}^{t} \big(Y_{k+\Tilde{l}}(s)-Y_{k}(s) \big) Y_{k}^{p-1}(s) \, \d s.
    \end{equation*}
    Adding the two equations above, we deduce that $\sum_{k} Y_{k}^p(t)$ satisfies
    $$ \sum_{k} Y_{k}^p(t) - \sum_{k} Y_{k}^p (0)= -4 \pi^2 C_{d} \kappa \, p  \sum_{l,i} \sum_{k} \theta_{l}^2 (a_{l,i} \cdot k)^2\!\! \int_0^t\!\! \big(Y_{k+l}(s)-Y_{k}(s) \big) \big(Y_{k+l}^{p-1}(s)-Y_{k}^{p-1}(s) \big) \,\d s.$$
Finally, noting that
  $$\sum_{i=1}^{d-1}(a_{l,i} \cdot k)^2 = |k|^2 - \frac{(k\cdot l)^2}{|k|^2} = |\Pi_l^\perp k|^2 ,$$
we obtain the desired result.
\end{proof}

By Lemma \ref{lem: sum-Fourier-ODE} and Gr\"onwall's inequality, in order to get the exponential decay of $\sum_{k} Y_{k}^p(t)$, we need
\begin{equation} \label{eq: estimate-for-gronwall}
    \sum_{k} Y_{k}^p \leq c \sum_{l} \sum_{k}\theta_{l}^2 \, |\Pi_l^\perp k|^2 (Y_{k+l}^{p-1}-Y_{k}^{p-1}) (Y_{k+l}-Y_{k})
\end{equation}
holds for some $c>0$. When $p=2$, the right-hand side of the above inequality can be seen as a Dirichlet form in $h^{1}(\Z_0^d)$. Unfortunately, we are not aware of the existence of such an inequality; nonetheless, Lemma \ref{lem: sequence sum} provides us with a similar estimate in one dimension. Thus, a natural approach to establish the Poincar\'e type inequality \eqref{eq: estimate-for-gronwall} is to decompose $\Z_0^d$ into countable orbits and apply Lemma \ref{lem: sequence sum} to each one of them. As the decomposition is quite complicated for $d \geq 3$, in this subsection we only provide the proof for $d=2$; cf. Appendix \ref{appendix: proof} for the higher dimensional case.

\begin{theorem} \label{thm: Fourier-coefficient-exp-decay}
    Let $Y(t):=\{Y_k(t)\}_k $. If $\supp \theta \subset \Z_0^d$ is compact and $u_0$ is smooth, then
    \begin{equation} \label{eq: Fourier-coefficient-exp-decay}
       \|Y(t)\|_{\ell^p} \leq e^{- \kappa C(\theta,d) \, \frac{1}{p} (1-\frac{1}{p})  \, t} \|Y(0)\|_{\ell^p} , \quad \forall \, t>0,
    \end{equation}
    where $C(\theta,d)$ is defined in \eqref{def: C-theta-d}.
\end{theorem}

\begin{proof}
     We will first provide a detailed proof for $d=2$ and then briefly explain the difference between $d=2$ and $d \geq 3$. The complete proof for higher dimensions can be found in Appendix \ref{appendix: proof}. The proof for $d=2$ is divided into the following several steps.

     \textbf{Step 1: constructing orbits.} Fixing a vector $l=(a,b) \in \Z_0^2$ and its orthogonal vector $l^{\perp}=(-b,a)$, we will use $l$ and $l^{\perp}$ to construct orbits such that they cover $\Z_0^2$ and have the uniform lower bound estimate \eqref{eq: estimate-orbit-0}. Firstly, we define the first quadrant $Q_1$, which consists of its interior part $\mathring{Q}_{1}$ and boundary sets $\partial Q_{1,1}$, $\partial Q_{1,2}$, defined as follows:
     \begin{align*}
        \mathring{Q}_{1} & =\{z=a(z) \, l + b(z) \, l^{\perp} \in \Z_{0}^2 ; \; a(z) , \, b(z) > 0\}, \\
        \partial Q_{1,1} & = \{z=b(z) \, l^{\perp} \in \Z_{0}^2 ; \; b(z) > 0\}, \\
        \partial Q_{1,2} & = \{z=a(z) \, l \in \Z_{0}^2 ; \;  a(z) > 0\},
    \end{align*}
    where $a(z) = \frac{z \cdot l}{|l|^2}$ and $b(z) = \frac{z \cdot l^{\perp}}{|l^{\perp}|^2}= \frac{z \cdot l^{\perp}}{|l|^2}$. For $i=2,3,4$, we can define the quadrants $Q_i$ by selecting two different base vectors from $\{ l,l^{\perp}, -l,-l^{\perp} \}$ and replacing $l$ and $l^{\perp}$ in the definition of $Q_1$. The corresponding interior set $\mathring{Q}_{i}$ and boundary sets $\partial Q_{i,1}$, $\partial Q_{i,2}$ are defined similarly.

    We introduce two distinct classes of orbits to cover $Q_{1}$. For any $z \in \mathring{Q}_{1} \cup \partial Q_{1,1}$, we define the first class of orbits starting from $z$ as
    \begin{equation*}
     \Gamma_{1,1}(z,n) = z + \Big\lfloor \frac{n+1}{2} \Big\rfloor \, l + \Big\lfloor \frac{n}{2} \Big\rfloor \, l^{\perp},  \quad n \in \N,
    \end{equation*}
    where $\lfloor x \rfloor$ is the largest integer that is less than or equal to $x$. Similarly, for any $z \in \mathring{Q}_{1} \cup \partial Q_{1,2} $, we define the second class of orbits starting from $z$ as
    \begin{equation*}
     \Gamma_{1,2}(z,n) = z + \Big\lfloor \frac{n}{2} \Big\rfloor \, l + \Big\lfloor \frac{n+1}{2} \Big\rfloor \, l^{\perp} ,  \quad n \in \N.
    \end{equation*}
    By the definition of $\mathring{Q}_{1}$, we know that $\Gamma_{1,1}(z,n)$ and $\Gamma_{1,2}(z,n)$ belong to $\mathring{Q}_{1}$ for any $n \geq 1$. The orbits $\Gamma_{i,1}(z,\cdot)$ and $\Gamma_{i,2}(z,\cdot)$ on $Q_i$ can be defined similarly to the orbits on $Q_1$; that is, $\Gamma_{i,1}(z,\cdot)$ and $\Gamma_{i,2}(z,\cdot)$ start from $z \in Q_{i}$ and move alternately along the base vectors of $Q_i$. We also have $\Gamma_{i,1}(z,n)$ and $\Gamma_{i,2}(z,n)$ belong to $\mathring{Q}_i$ for any $n \geq 1$.

    Let $I_{1,1} := \{z \in \mathring{Q}_{1} ; \, z-l^{\perp} \notin Q_1 \} \cup \partial Q_{1,1}$ and $I_{1,2} := \{z \in \mathring{Q}_{1} ; \, z-l \notin Q_1 \} \cup \partial Q_{1,2}$. We assert that the orbits starting from $I_{1,1} $ and $I_{1,2}$ cover $\mathring{Q}_{1}$ exactly twice (see details in Step 3). Figure \ref{fig: decomposition} provides a clear example of the concepts we have defined so far.
    \begin{figure}[ht]
    \centering
    \includegraphics[width=15cm,height=10cm]{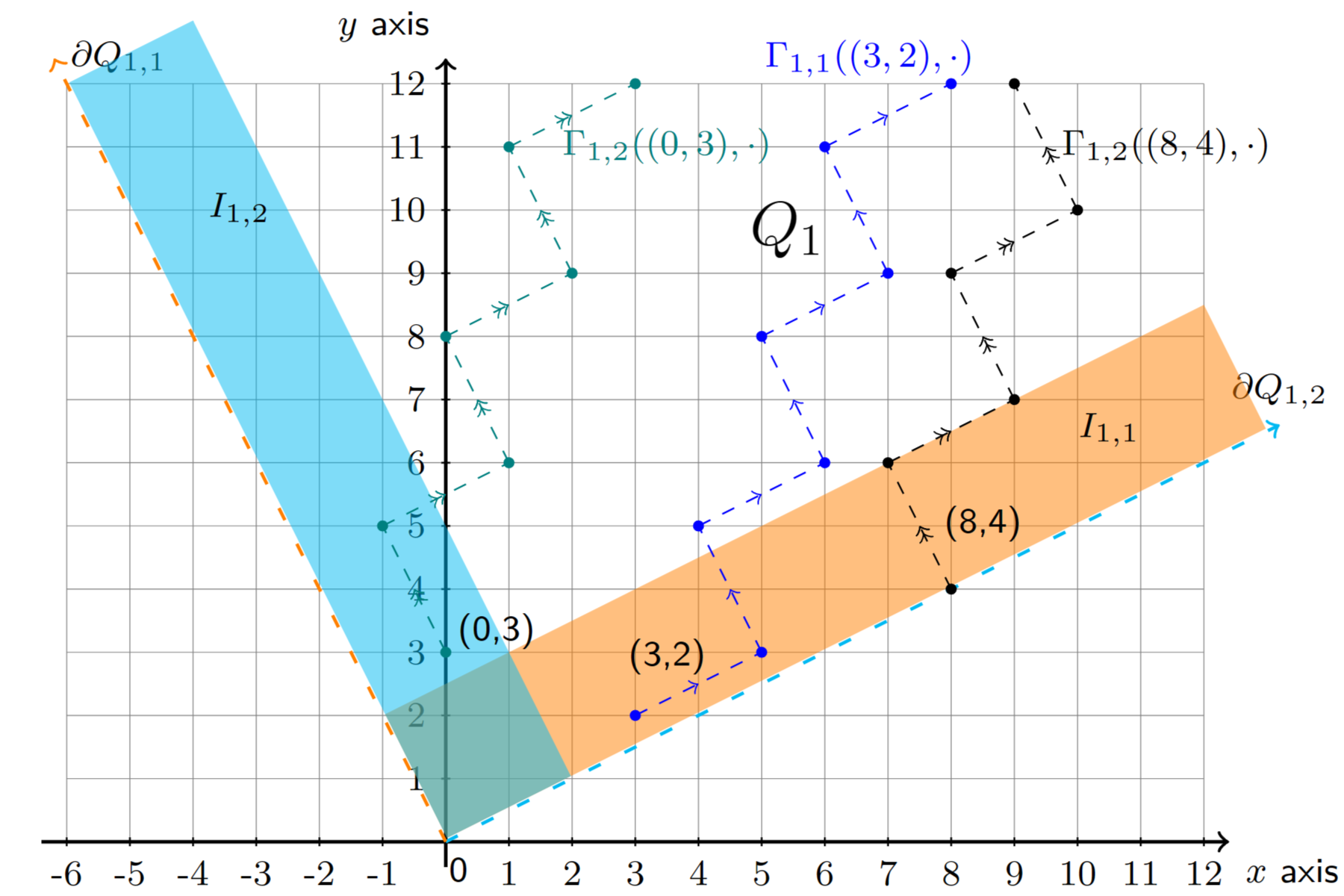}
    \caption{The decomposition of $Q_1$ for $l=(2,1)$} \label{fig: decomposition}
    \end{figure}

    \textbf{Step 2: estimates of orbits.} Let $\O_{1} = \{\Gamma_{1,1}(z,\cdot); z \in \mathring{Q}_{1} \cup \partial Q_{1,1} \} \cup \{\Gamma_{1,2}(z,\cdot); z \in \mathring{Q}_{1} \cup \partial Q_{1,2} \}$ be the collection of orbits on quadrant $Q_1$. Recall that $Y_k= \E|\hat u_k|^2,\, k\in \Z^2_0$. We assert that the orbit $O(\cdot) \in \mathcal{O}_{1}$ satisfies, for any $p>1$,
    \begin{equation} \label{eq: estimate-orbit-0}
        \sum_{n \in \N} |\Pi^{\perp}_{O(n+1)-O(n)} O(n)|^2 \big(Y_{O(n+1)}^{p-1}-Y_{O(n)}^{p-1}\big) \big(Y_{O(n+1)}-Y_{O(n)} \big) \geq \frac{p-1}{8 p^2 |l|^2 } \sum_{n \in \N} Y_{O(n)}^{p} .
    \end{equation}
    For simplification, we denote $\Delta O(n)= O(n+1)-O(n)$ and $ \Pi^{\perp}_{\Delta} O(n):=\Pi^{\perp}_{O(n+1)-O(n)} O(n)$ for all orbit $O(\cdot) \in \O_1$ and any $n \in \N$. By Lemma \ref{lem: sequence sum}, to prove the estimate \eqref{eq: estimate-orbit-0}, we only need to show
    \begin{equation} \label{eq: estimate-orbit-project-0}
        |\Pi^{\perp}_{\Delta} O(n)|^2= |\Pi^{\perp}_{O(n+1)-O(n)} O(n)|^2 \geq \frac{(n+1)^2}{4 |l|^2}, \quad \forall  n \in  \N.
    \end{equation}
    Recall the construction of the orbits; we know that $\Delta O(n)$ equals $l$ or $l^{\perp}$. We first prove the inequality \eqref{eq: estimate-orbit-project-0} for the first class of orbits $\Gamma_{1,1}(z,\cdot)$. When $n$ is even, we have
    \begin{equation*}
        \big|\Pi^{\perp}_{\Delta} O(n) \big|^2 = \big| \Pi^{\perp}_{l} \Gamma_{1,1}(z,n) \big|^2 =\big|\Pi^{\perp}_{l} z \big|^2+n \, \Pi^{\perp}_{l} z \cdot l^{\perp} + \frac{n^2}{4} \big|l^{\perp} \big|^2 .
    \end{equation*}
    Due to $z=a(z) \, l+ b(z) \, l^{\perp} \in \mathring{Q}_{1} \cup \partial Q_{1,1}$, we have $b(z)>0$, then $\Pi^{\perp}_{l} z \cdot l^{\perp} =  b(z) \, \big| l^{\perp} \big|^2 > 0$. Since $z,l$, and $l^{\perp}$ belong to $\Z_0^2$, we can conclude that $\Pi^{\perp}_{l} z \cdot l^{\perp}= z \cdot l^{\perp}$ is a positive integer. This implies that $\Pi^{\perp}_{l} z \cdot l^{\perp} \geq 1$. Similarly, due to $b(z)>0$, we have
    \begin{equation*}
       |\Pi^{\perp}_{l} z|^2 = \Big( z - \frac{z \cdot l}{|l|^2} l\Big) \cdot \Big( z - \frac{z \cdot l}{|l|^2} l\Big)  = \frac{|z|^2|l|^2-(z\cdot l)^2}{|l|^2} \geq \frac{1}{|l|^2}.
    \end{equation*}
    Thus, for even $n\in \N$,
    \begin{equation*}
        |\Pi^{\perp}_{\Delta} O(n)|^2 \geq \frac{1}{|l|^2} \Big(1+n+\frac{n^2}{4} \Big) \geq \frac{(n+1)^2}{4|l|^2 }.
    \end{equation*}
    When $n$ is odd, similar to the estimates as $n$ is even, we have
    \begin{equation*}
        |\Pi^{\perp}_{\Delta} O(n)|^2 = \big|\Pi^{\perp}_{l^{\perp}} z \big|^2 + (n+1) \, \Pi^{\perp}_{l^{\perp}} z \cdot l + \frac{(n+1)^2}{4} \big| l \big|^2 \geq  \frac{(n+1)^2}{4|l|^2 }.
    \end{equation*}
    Combining both estimates above, we obtain that inequality \eqref{eq: estimate-orbit-project-0} holds for all the first class orbits. In the same way, inequality \eqref{eq: estimate-orbit-project-0} also holds for the second class orbits. Using Lemma \ref{lem: sequence sum}, we can derive the inequality \eqref{eq: estimate-orbit-0} for all $O(\cdot) \in \O_{1}$. Similarly, the inequality \eqref{eq: estimate-orbit-0} also holds for all orbits on the quadrant $Q_i$, where $i=2,3,4$.

    \textbf{Step 3: covering $\Z_0^2$ exactly twice.} Let $\Bar{\O}_{1} = \{\Gamma_{1,1}(z,\cdot); z \in I_{1,1} \} \cup \{\Gamma_{1,2}(z,\cdot); z \in I_{1,2} \}$ be the set of all the orbits starting from $I_{1,1}$ and $I_{1,2}$. Recall the definitions of
    \begin{align*}
        I_{1,1} & = \{z \in \mathring{Q}_{1} ; \, z-l^{\perp} \notin Q_1 \} \cup \partial Q_{1,1}, \\
        I_{1,2} & = \{z \in \mathring{Q}_{1} ; \, z-l \notin Q_1 \} \cup \partial Q_{1,2}.
    \end{align*}
    For any $z \in \partial Q_{1,1} \cup \partial Q_{1,2}$, only one orbit in $ \Bar{\O}_1$ passes $z$, and it is the orbit starting from $z$. Taking into account the fact that the boundary sets will be counted twice in different quadrants, to prove that the orbits in $\cup_{i=1}^4 \Bar{\O}_i$ cover $\Z_0^2$ exactly twice, it suffices to prove that the orbits belonging to $\Bar{\O}_i$ (defined similarly as $\Bar{\O}_1$) cover $\mathring{Q}_{i}$ exactly twice.

     We say that $O_1 \preceq O_2$ if there exists $m \in \N$ such that $O_1(\cdot) = O_2(\cdot + m)$. Then $(\O_{1},\preceq)$ is a partially ordered set, and we assert that $\Bar{\O}_{1}$ is the set of all the maximal element of $(\O_{1},\preceq)$. For any orbit $O(\cdot) \in \O_1$ and $m \in \N$, we can define the orbit $O(m, \cdot)$ as follows:
    \begin{equation*}
        O(m, n)= \begin{cases}
        O(0)+ \lfloor \frac{n-m+1}{2} \rfloor \big(O(1)-O(0)\big) + \lfloor \frac{n-m}{2} \rfloor \big(O(2)-O(1)\big), \quad n \geq m; \\
        O(0) - \lfloor \frac{m-n}{2} \rfloor \big(O(1)-O(0)\big) - \lfloor \frac{m-n+1}{2} \rfloor \big(O(2)-O(1)\big), \quad n \leq m-1.
        \end{cases}
    \end{equation*}
    Let $m_0$ be the largest integer such that $O(m_0,n) \in \O_{1}$. It is clear that $m_0$ is finite and $O(m_0,\cdot)$ is the only maximal element corresponding to $O(\cdot)$. If $O(m_0,1)-O(m_0,0)=l$, then $O(m_0,0) \in \partial Q_{1,1}$ or $O(m_0,0) \in \mathring{Q}_1$ but $O(m_0,0) - l^{\perp} \notin Q_1$, so $O(m_0,0) \in I_{1,1}$ and $O(m_0,\cdot) \in \Bar{\O}_1$. Similarly, if $O(m_0,1)-O(m_0,0)=l^{\perp}$, then $O(m_0,\cdot)$ also belongs to $\Bar{\O}_1$.

    It is clear that for any $O_1, O_2 \in \Bar{\O}_1$, if $O_1(m)=O_2(n) $ and $O_1(m+1)=O_2(n+1)$ for some $m,n \in \N$, then $O_1=O_2$. Then using this property and the fact that $\Bar{\O}_{1}$ is the set of all maximal elements in $(\O_{1},\preceq)$, we know that the number of orbits in $\Bar{\O}_1$ passing $z \in \mathring{Q}_{1}$ is equal to the number of the orbits in $\O_1$ starting from $z$. So orbits in $\Bar{\O}_1$ cover $\mathring{Q}_{1}$ exactly twice. As the above discussions, the orbits in $\cup_{i=1}^4 \Bar{\O}_i$ cover $\Z_0^2$ exactly twice.

    \textbf{Step 4: the final step.} Applying inequality \eqref{eq: estimate-orbit-0} to all orbits in $\cup_{i=1}^4 \Bar{\O}_i$ and using the fact that the orbits in $\cup_{i=1}^4 \Bar{\O}_i$ cover $\Z_0^2$ exactly twice, we obtain
    \begin{equation} \label{eq: estimate-orbit-cover-0}
        \sum_{i=1}^{4} \sum_{O \in \Bar{\O}_{i}} \sum_{n \in \N} |\Pi^{\perp}_{\Delta} O(n)|^2 \big(Y_{O(n+1)}^{p-1}-Y_{O(n)}^{p-1}\big) \big(Y_{O(n+1)}-Y_{O(n)} \big)
        \geq \frac{p-1}{4 p^2 |l|^2 } \sum_{k \in \Z_0^2} Y_{k}^{p}.
    \end{equation}
    We denote the summands of the right-hand side of inequality \eqref{eq: estimate-for-gronwall} as
    \begin{equation*}
         S_{k,l}:=  \theta_{l}^2 \,  \big| \Pi^{\perp}_{l} k \big|^2  \big(Y_{k+l}^{p-1}-Y_{k}^{p-1} \big) \big(Y_{k+l}-Y_{k} \big).
    \end{equation*}
   Note that the orbits in $\cup_{i=1}^4 \Bar{\O}_i$ only going in two directions (like orbits in $\Bar{\O}_1$ only moving along $l$ and $l^{\perp}$, but not $-l$ or $-l^{\perp}$); using the fact that $S_{k,l}=S_{k+l,-l}$, we have
    \begin{align*}
        & \sum_{k \in \Z_0^2 } \, \sum_{h=l,l^{\perp},-l,-l^{\perp} } \theta_{h}^2 \, \big|\Pi^{\perp}_{h} k \big|^2  \big(Y_{k+h}^{p-1}-Y_{k}^{p-1} \big) \big(Y_{k+h}-Y_{k} \big) \\
        & \quad \geq 2 \, \theta_{l}^2 \sum_{i=1}^{4} \sum_{O \in \Bar{\O}_{h,i}} \sum_{n \in \N} |\Pi^{\perp}_{\Delta} O(n)|^2 \big(Y_{O(n+1)}^{p-1}-Y_{O(n)}^{p-1}\big) \big(Y_{O(n+1)}-Y_{O(n)} \big).
    \end{align*}
    Using inequality \eqref{eq: estimate-orbit-cover-0}, we obtain
    \begin{align*}
        & \sum_{l \in \Z_0^2} \sum_{k \in \Z_0^2 } \, \theta_{l}^2 \, \big|\Pi^{\perp}_{l} k \big|^2  \big(Y_{k+l}^{p-1}-Y_{k}^{p-1} \big) \big(Y_{k+l}-Y_{k} \big) \\
        & = \frac{1}{4} \sum_{l \in \Z_0^2} \sum_{k \in \Z_0^2 } \, \sum_{h=l,l^{\perp},-l,-l^{\perp} } \theta_{h}^2 \, \big|\Pi^{\perp}_{h} k \big|^2  \big(Y_{k+h}^{p-1}-Y_{k}^{p-1} \big) \big(Y_{k+h}-Y_{k} \big) \\
        &  \geq \frac{p-1}{8 p^2 } \sum_{l \in \Z_0^2} \frac{\theta_{l}^2}{|l|^2} \sum_{k \in \Z_0^2} Y_{k}^{p}=  \frac{p-1}{8 p^2 } \| \theta \|_{h^{-1}}^2 \sum_{k \in \Z_0^2} Y_{k}^{p}.
    \end{align*}
    By Lemma \ref{lem: sum-Fourier-ODE} and Gr\"onwall inequality, we get inequality \eqref{eq: Fourier-coefficient-exp-decay} when $d=2$.

    \textbf{The difference between $d=2$ and $d \geq 3$.} When $d=2$, we can conveniently choose $l^{\perp}$ for any fixed $l \in \Z_0^2$ and use them to construct orbits. However, due to the more intricate structure of $\Z_0^d$ when $d \geq 3$, typically we can only choose $l_2 \in \Z_0^d$ that are linearly independent with fixed $l_1 \in \Z_0^d$. Recall the notations in Section \ref{subsec: setting}. For $d \geq 3$, we will decompose $\Z_0^d$ into $\cup_{h \in P_Z^{\perp}(l_1,l_2)} \S_{h}(l_1,l_2)$ and handle each plane $\S_{h}(l_1,l_2)$ similarly as in the case of $d=2$. When $h \in P_Z^{\perp}(l_1,l_2) \cap \Z_0^d$, we need to treat it additionally because the estimate when $d=2$ does not include it. Since $l_1$ and $l_2$ are not orthogonal, the decomposition and estimation will be more complicated than $d=2$.
\end{proof}

Now we apply Theorem \ref{thm: Fourier-coefficient-exp-decay} and H\"older's inequality to prove the mixing properties of solutions to stochastic transport equation \eqref{STE-Ito} in the average sense. Firstly, we will prove the mixing property of $u$ in the good case, then, using Lemma \ref{lem: convegence}, we get the mixing property in the general case by an approximation argument.

\begin{proof}[Proof of Theorem \ref{thm: mixing-exp}]

We first consider the good case that $\supp \theta$ is compact and the initial data $u_0$ is smooth. In this case, the solution $u$ to \eqref{STE-Ito} satisfies the conditions of Theorem \ref{thm: Fourier-coefficient-exp-decay}.

When $0<\beta \leq \frac{d}{4}$, take $1<p < \frac{d}{d-2\beta}$, $q =\frac{p}{p-1}> \frac{d}{2\beta}$. By H\"older inequality, it holds
\begin{align*}
    \E \| u(t) \|_{H^{-\beta}}^2 & = \E \sum_{k} \frac{|\hat{u}_{k}(t)|^2}{|2 \pi k|^{2\beta}} \leq \Big(\sum_{k} \big(\E |\hat{u}_{k}(t)|^2 \big)^p \Big)^{1/p} \Big(\sum_{k} \frac{1}{|2 \pi k|^{2\beta \,q}}\Big)^{1/q} \\
    & \leq C e^{-\kappa C(\theta,d) \frac{p-1}{p^2} t} \Big( \sum_{k} |\hat{u}_{k}(0)|^{2p} \Big)^{1/p},
\end{align*}
where the last step is due to Theorem \ref{thm: Fourier-coefficient-exp-decay} and the constant $C(\theta,d)$ is defined in \eqref{def: C-theta-d}. For any $\epsilon \in (0,\frac{\beta(d-2\beta)}{d^2})$, choosing $p$ such that $ \frac{1}{p}=\frac{d-2\beta}{d}+\epsilon<1$, we obtain
\begin{align*}
    \E \| u(t) \|_{H^{-\beta}}^2 & \leq C e^{-\kappa C(\theta,d) \, \left( \frac{1}{4}-(\frac{d-2\beta}{d}+\epsilon-\frac{1}{2})^2 \right) t } \, \sup_{k} |\hat{u}_{k}(0)|^{\frac{2p-2}{p}}  \, \Big(\sum_{k} |\hat{u}_{k}(0)|^{2} \Big)^{1/p} \\
    & \leq  C e^{-2 \kappa C(\theta,d) \, \big(\frac{\beta (d-2\beta)}{d^2}-\epsilon \big) \, t} \, \|u(0)\|_{L^2}^{2},
\end{align*}
where the last step is due to the fact $\sup_{k} |\hat{u}_{k}(0)|^{2} \leq \sum_{k} |\hat{u}_{k}(0)|^{2}= \|u_0\|_{L^2}^{2}$.

When $\beta>\frac{d}{4}$, $\sum_{k} \frac{1}{|k|^{4\beta}}<+\infty$. Using H\"older inequality and Theorem \ref{thm: Fourier-coefficient-exp-decay}, we obtain
\begin{align*}
    \E \| u(t) \|_{H^{-\beta}}^2 & = \E \sum_{k} \frac{|\hat{u}_{k}(t)|^2}{|2 \pi k|^{2\beta}} \leq  \Big(\sum_{k} \frac{1}{|2 \pi k|^{4\beta}} \Big)^{1/2} \Big( \sum_{k} \big(\E|\hat{u}_{k}(t)|^2 \big)^2 \Big)^{1/2} \\
    & \leq C e^{- \kappa C(\theta,d) \,t /4} \Big( \sum_{k} |\hat{u}_{k}(0)|^4 \Big)^{1/2} .
\end{align*}
Due to $\sum_{k} |\hat{u}_{k}(0)|^4 \leq \sup_{k} |\hat{u}_{k}(0)|^{2} \sum_{k} |\hat{u}_{k}(0)|^{2} \leq \|u_0\|_{L^2}^{4}$, we get the desired result.

For the general case, we will choose a sequence $\{u^n\}$ satisfying the above estimates and converging to the solution $u$ of \eqref{STE-Ito}. Let $u_0^n \in C^{\infty}(\T^d)$ and $u_0^n$ converge to $u_0$ in $L^2(\T^d)$ as $n$ tends to $+\infty$. For any integer $n>0$ and vector $l \in \Z_0^d$, we set
    $$\theta^{(n)}_{l}=\frac{1}{ \big(\sum_{|l| \leq n } \theta_{l}^2 \big)^{1/2} } \, \theta_{l} \, 1_{ \{|l| \leq n\}},$$
    then $\theta^{(n)}_{l}$ converges to $\theta_{l}$ point-wise. Let $u^n$ be the solution to
    \begin{equation*}
        \d u^n = \kappa \Delta u^n \d t + \sqrt{C_d \kappa} \sum_{k,i} \theta_{k}^{(n)} e_{k} a_{k,i} \cdot \nabla u^n \d W_{t}^{k,i}, \quad u^n(0)=u_0^n.
    \end{equation*}
    Then Theorem \ref{thm: mixing-exp} (i) and (ii) hold for $u^n$ with constant $C(\theta^{(n)},d)$; it is clearly that $C(\theta^{(n)},d)\to C(\theta,d)$ as $n\to \infty$. Using the pathwise uniqueness of equation \eqref{STE-Ito} (see \cite[Theorem 3]{Gal20}), we deduce that there exists a subsequence $\{u^{n_j}\}$ converging weakly to the solution $u$ of equation \eqref{STE-Ito} in $L^2(\Omega; L^2([0,T]; L^2(\T^d)))$ and also in $L^2(\Omega; L^2([0,T]; H^{-\beta}(\T^d)))$. Observe that the collection of processes in $L^2(\Omega; L^2([0,T]; H^{-\beta}(\T^d)))$ satisfying Theorem \ref{thm: mixing-exp} (i) (or (ii)) is a convex, closed subset. Therefore, it is also weakly closed (see Lemma \ref{lem: convegence}), which implies that Theorem \ref{thm: mixing-exp} (i) (or (ii)) also holds for the weak limit $u$.
\end{proof}

We finish this subsection with some heuristic discussions, trying to show that direct computation of negative Sobolev norms may not work for general noises.

\begin{remark}\label{rem-direct-computation}
Consider $d=2$; we deduce from equation \eqref{eq: Fourier ODE-1} and Lemma \ref{lem: matrix-property} that
\begin{equation*}
    \E \|u(t)\|_{H^{-1}}^2 - \|u_0\|_{H^{-1}}^2
    =  4 \kappa  \sum_k \sum_{l} \theta_{l}^2 \frac{(a_{l} \cdot k)^2}{|k|^2} \int_{0}^{t} \big( \E |\hat{u}_{k-l}(s)|^2 -\E |\hat{u}_{k}(s)|^2 \big) \, \d s.
\end{equation*}
Similarly to the treatment of $I_1(t)$ in the proof of Lemma \ref{lem: uniform-bounded} below, we have
\begin{align*}
    \E \|u(t)\|_{H^{-1}}^2 - \|u_0\|_{H^{-1}}^2
    & =  2\kappa  \sum_k \sum_{l} \theta_{l}^2 (a_{l} \cdot k)^2 \Big(\frac{1}{|k+l|^2} + \frac{1}{|k-l|^2} - \frac{2}{|k|^2}\Big) \, \int_{0}^{t}  \E |\hat{u}_{k}(s)|^2 \, \d s \\
    & = 2\kappa  \sum_k \sum_{l} \theta_{l}^2 \frac{(a_{l} \cdot k)^2}{|k|^2} \Big(\frac{ 2(|k|^2+|l|^2) |k|^2} {(|k|^2+|l|^2)^2-4 (l \cdot k)^2 }-2 \Big) \, \int_{0}^{t}  \E |\hat{u}_{k}(s)|^2 \, \d s.
\end{align*}
Consider the special case in which $\supp\, \theta= \{(\pm 1, 0), (0,\pm 1)\}$ consists of four points and $\theta_l=\frac{1}{2}$ for every $l\in \supp\, \theta$; then for any $k=(k_0,k_0) \in \Z_0^2$, we have
\begin{equation*}
    \sum_{l} \theta_{l}^2 \frac{(a_{l} \cdot k)^2}{|k|^2} \Big(\frac{2(|k|^2+|l|^2)|k|^2}{(|k|^2+|l|^2)^2-4 (l \cdot k)^2 }-2 \Big) = \frac{4k_0^4 +2k_0^2}{4 k_0^4+1} -1 \geq \frac{1}{2 k_0^2+3}>0.
\end{equation*}
Suppose we are in the extreme setting where the energy is totally concentrated on the diagonal part, that is
\begin{equation*}
    \Big\{k \in \Z_0^2; \ \int_{0}^{t} \E |\hat{u}_{k}(s)|^2 \, \d s \neq 0 \Big\} \subset A_{diag}:=\{k=(k_0,k_0); \ k_0 \in \Z \},
\end{equation*}
then the expectation $\E \|u(t)\|_{H^{-1}}^2$ is increasing. If $\{k \in \Z_0^2; \,|\hat{u}_{k}(0)|^2 \neq 0 \} \subset A_{diag}$, it is reasonable to expect that the distribution of energy spectrum will not change too much in a short time, which is close to $A_{diag}$, and thus $\E \|u(t)\|_{H^{-1}}^2$ will increase during this short time.
\end{remark}

\subsection{Mixing under $\P$-almost sure sense} \label{subsec: mixing-P-a.s.}

This section will discuss the mixing properties of solution $u$ to \eqref{STE-Ito} in $\P$-a.s. sense. First, we will prove that, for some small $t_0>0$, the map $F(u_0) :=\{u(t)\}_{t\in [0,t_0]} $ is bounded from $H^{-1}(\T^d)$ to $C([0,t_0]; H^{-1}(\T^d))$. Then, using the Markov property of $u$ and the mixing in the average sense, we obtain a uniform decay rate of $\|u\|_{H^{-1}}^2$ over the time intervals $[n t_0,(n+1) t_0]$ for all $n \geq 0$. Finally, we prove the mixing properties in $\P$-a.s sense by Borel-Cantelli lemma.

\begin{lemma} \label{lem: uniform-bounded}
    If initial data $u_0 \in L^2(\T^d)$ and noise coefficient $\theta$ satisfies $ \| \theta \|_{h^1}^2 < + \infty$, then the solution $u$ to the stochastic transport equation \eqref{STE-Ito} satisfies
    \begin{equation*}
        \E \Big[ \sup_{t \in [0,t_0]} \| u(t) \|_{H^{-1}}^2 \Big] \leq 2 \|u_0\|_{H^{-1}}^2,
    \end{equation*}
    where $t_0= \big(\frac{\sqrt{11}-3}{16} \big)^2 \frac{1}{ \pi^2 \, d \, \kappa \, \| \theta \|_{h^1}^2 }>0$.
\end{lemma}

\begin{proof}
    Firstly, we consider the good case where $\supp \theta \subset \Z_0^d$ is compact and the initial data $u_0$ is smooth. Recall the equation \eqref{eq: Fourier-coefficient-norm} and Lemma \ref{lem: matrix-property}, we obtain
    \begin{align*}
    \sum_k  \frac{|\hat{u}_{k}(t)|^2 }{|2 \pi k|^2} - \|u_0\|_{H^{-1}}^2
    &=  2 C_d \kappa  \sum_k \sum_{l,i} \theta_{l}^2 \frac{(a_{l,i} \cdot k)^2}{|k|^2} \int_{0}^{t} \big( |\hat{u}_{k-l}(s)|^2 -|\hat{u}_{k}(s)|^2 \big) \, \d s   \\
    & \quad +  \frac{{\rm i}}{2 \pi} \sqrt{C_d \kappa} \sum_k \int_{0}^{t}  \sum_{l,i} \theta_l \frac{a_{l,i} \cdot k}{|k|^2} ( \overline{\hat{u}_k} \hat{u}_{k-l}-\hat{u}_k \overline{\hat{u}_{k+l}} ) \, \d W_{s}^{l,i}.
    \end{align*}
    We denote the two terms on the right-hand side of the above equation as $I_1(t)$ and $I_2(t)$, respectively. For fixed $l \in \Z_0^d$, substituting $k$ for $\Tilde{k}+l$ in $I_1(t)$ yields
    \begin{equation*}
        I_1(t) = 2 C_d \kappa  \sum_{\Tilde{k}} \sum_{l,i} \theta_{l}^2 \frac{(a_{l,i} \cdot \Tilde{k})^2}{|\Tilde{k}+l|^2} \int_{0}^{t} \big( |\hat{u}_{\Tilde{k}}(s)|^2 -|\hat{u}_{\Tilde{k}+l}(s)|^2 \big) \, \d s.
    \end{equation*}
    Substituting $l$ for $\Tilde{l}=-l$ in $I_1(t)$ and adding the two equations, we obtain
    \begin{align*}
         2 I_1(t) & = 2 C_d \kappa \int_0^t  \sum_{k} \sum_{l,i} \theta_{l}^2 (a_{l,i} \cdot (k+l))^2 \, \Big( \frac{1}{|k|^2}- \frac{1}{|k+l|^2} \Big) |\hat{u}_{k+l}(s)|^2 \d s \\
         & \quad + 2 C_d \kappa \int_0^t \sum_{l,i} \sum_{k} \theta_{l}^2 (a_{l,i} \cdot k)^2 \, \Big( \frac{1}{|k+l|^2}-\frac{1}{|k|^2} \Big) |\hat{u}_{k}(s)|^2 \d s.
    \end{align*}
    Substituting $k$ for $k-l$ in the first term of the right hand of the above identity, we get
    \begin{align*}
        I_1(t) & = 2 C_d \kappa \int_{0}^{t} \sum_k \sum_{l,i} \theta_{l}^2 (a_{l,i} \cdot k)^2 \Big(\frac{1}{|k+l|^2} + \frac{1}{|k-l|^2} - \frac{2}{|k|^2}\Big) |\hat{u}_{k}(s)|^2  \d s \\
        & = 2 C_d \kappa \int_{0}^{t} \sum_k \sum_{l,i} \theta_{l}^2 (a_{l,i} \cdot k)^2 \Big(\frac{-2k \cdot l - |l|^2}{|k|^2 |k+l|^2} + \frac{2k \cdot l - |l|^2}{|k|^2 |k-l|^2} \Big) |\hat{u}_{k}(s)|^2  \d s \\
        & \leq 2 C_d \kappa \int_{0}^{t} \sum_k \sum_{l,i} \theta_{l}^2 (a_{l,i} \cdot k)^2 \frac{8 (k \cdot l)^2}{|k|^2 |k+l|^2|k-l|^2} |\hat{u}_{k}(s)|^2  \d s.
    \end{align*}
    Due to $\frac{(a_{l,i} \cdot k)^2}{|k+l|^2|k-l|^2} \leq \min \{\frac{1}{|k+l|^2} , \frac{1}{|k-l|^2} \} \leq \frac{1}{|k|^2}$, we get
    \begin{align}
        \sup_{t \in [0,t_0]} I_1(t) & \leq 16 \, d \, \kappa  \| \theta \|_{h^1}^2 \sup_{t \in [0,t_0]}  \int_{0}^{t} \sum_{k} \frac{|\hat{u}_{k}(s)|^2}{|k|^2} \d s \nonumber \\
        & \leq 64 \pi^2 \, d \, \kappa   \| \theta \|_{h^1}^2 \big( \sup_{t \in [0,t_0]} \|u(t)\|_{H^{-1}}^2 \big) \, t_0 . \label{eq: uniform-bounded-1}
    \end{align}

    For fixed $l \in \Z_0^d$, let $k=\Tilde{k}+l$ and $k=\Bar{k} - l$, then the stochastic Fubini theorem yields
    \begin{align*}
        I_2(t) & =  \frac{{\rm i}}{2 \pi} \sqrt{C_d \kappa} \sum_{l,i} \theta_l \sum_{\Tilde{k}} (a_{l,i} \cdot \Tilde{k})  \int_{0}^{t} \frac{1}{|\Tilde{k}+l|^2} \hat{u}_{\Tilde{k}} \overline{\hat{u}_{\Tilde{k}+l}} \, \d W_{s}^{l,i} \\
        & \quad - \frac{{\rm i}}{2 \pi} \sqrt{C_d \kappa}  \sum_{l,i} \theta_l \sum_{\Bar{k}} (a_{l,i} \cdot \Bar{k})  \int_{0}^{t} \frac{1}{|\Bar{k}-l|^2}   \overline{\hat{u}_{\Bar{k}}} \hat{u}_{\Bar{k}-l} \, \d W_{s}^{l,i}.
    \end{align*}
    Adding the two different formula of $I_2(t)$, we obtain $I_2$ equals
    \begin{equation*}
        \frac{i}{4 \pi} \sqrt{C_d \kappa} \sum_k \sum_{l,i} \theta_l (a_{l,i} \cdot k)  \int_{0}^{t} \Big( \Big( \frac{1}{|k|^2}-\frac{1}{|k-l|^2} \Big)  \overline{\hat{u}_k} \hat{u}_{k-l}- \Big( \frac{1}{|k|^2}-\frac{1}{|k+l|^2} \Big) \hat{u}_k \overline{\hat{u}_{k+l}} \Big) \, \d W_{s}^{l,i}.
    \end{equation*}
    In the same way (taking $k=\Tilde{k}+l$ and $k=\Bar{k} - l$, and adding them), we have
    \begin{equation*}
        \sum_k \sum_{l,i} \theta_l (a_{l,i} \cdot k)  \int_{0}^{t} \Big( \frac{|l|^2}{|k|^2|k-l|^2} \overline{\hat{u}_k} \hat{u}_{k-l}- \frac{|l|^2}{|k|^2|k+l|^2} \hat{u}_k \overline{\hat{u}_{k+l}} \Big) \, \d W_{s}^{l,i} =0.
    \end{equation*}
    Combining the above equations, we get
    \begin{equation*}
        I_2(t) = - \frac{{\rm i}}{2 \pi} \sqrt{C_d \kappa} \sum_k \sum_{l,i} \theta_l \frac{(a_{l,i} \cdot k)(l \cdot k)}{|k|^2} \int_{0}^{t} \Big( \frac{1}{|k-l|^2}  \overline{\hat{u}_k} \hat{u}_{k-l}+ \frac{1}{|k+l|^2} \hat{u}_k \overline{\hat{u}_{k+l}} \Big) \, \d W_{s}^{l,i}.
    \end{equation*}

    Then by Burkholder-Davis-Gundy inequality,
    \begin{align*}
        & \E \sup_{t \in [0,t_0]} |I_2 (t)|   \leq 6 \, \E \sup_{t \in [0,t_0]} [I_2,I_2]_t^{1/2} \\
        &  \quad = \frac{3}{\pi} \sqrt{C_d \kappa} \, \E \left[  \sum_{l,i}  \theta_l^2 \int_{0}^{t_0} \Big|   \sum_k  \frac{(a_{l,i} \cdot k)(l \cdot k)}{|k|^2} \Big( \frac{1}{|k-l|^2}  \overline{\hat{u}_k} \hat{u}_{k-l}+ \frac{1}{|k+l|^2} \hat{u}_k \overline{\hat{u}_{k+l}} \Big) \Big|^2  \, \d s \right]^{1/2} \\
        & \quad \leq  \frac{6}{\pi} \sqrt{C_d \kappa} \, \E \left[  \sum_{l,i}  \theta_l^2 \int_{0}^{t_0} \Big(   \sum_k  \frac{(a_{l,i} \cdot k)(l \cdot k)}{|k|^2|k-l|^2}  |\hat{u}_k| |\hat{u}_{k-l}| \Big)^2  \, \d s \right]^{1/2} ,
    \end{align*}
    where in the last passage we used the inequality $(a+b)^2 \leq 2 a^2+ 2 b^2$ and the symmetry in $l$. Applying Cauchy inequality, it holds
    \begin{align}
        \E \sup_{t \in [0,t_0]} |I_2(t)| & \leq  \frac{6}{\pi} \sqrt{C_d \kappa} \E \left[  \sum_{l,i}  \theta_l^2  \int_{0}^{t_0} \Big( \sum_k  \frac{(l \cdot k)^2}{|k|^2} \frac{|\hat{u}_k|^2}{|k|^2} \Big) \Big( \sum_k  \frac{(a_{l,i} \cdot k)^2}{|k-l|^2} \frac{|\hat{u}_{k-l}|^2}{|k-l|^2} \Big) \, \d s \right]^{1/2} \nonumber \\
        & \leq \frac{6}{\pi} \sqrt{d \kappa} \E \left[  \sum_{l}  |l|^2 \theta_l^2  \int_{0}^{t_0} \Big( \sum_k \frac{|\hat{u}_k(s)|^2}{|k|^2} \Big)^2  \, \d s \right]^{1/2} \nonumber \\
        & \leq 24 \pi \sqrt{d \kappa} \, \big(\sum_{l}  |l|^2 \theta_l^2 \big)^{1/2} \, \E \Big[ \sup_{t \in [0,t_0]} \| u(t) \|_{H^{-1}}^2 \Big] \, t_0^{1/2}.  \label{eq: uniform-bounded-2}
    \end{align}

    Recall the choice of $t_0$, and combining estimates \eqref{eq: uniform-bounded-1} and \eqref{eq: uniform-bounded-2}, we obtain
    \begin{align*}
        \E \Big[ \sup_{t \in [0,t_0]} \| u(t) \|_{H^{-1}}^2 \Big] & \leq \E \Big[ \sup_{t \in [0,t_0]} |I_1(t)| \Big] + \E \Big[ \sup_{t \in [0,t_0]} |I_2(t)| \Big] +\|u_0\|_{H^{-1}}^2 \\
        & \leq \frac{1}{2} \, \E \Big[ \sup_{t \in [0,t_0]} \| u(t) \|_{H^{-1}}^2 \Big] + \|u_0\|_{H^{-1}}^2.
    \end{align*}
    So $\E \big[ \sup_{t \in [0,t_0]} \| u(t) \|_{H^{-1}}^2 \big] \leq 2 \|u_0\|_{H^{-1}}^2$, we get the desired result in the good case.

    For general case, we assume the initial data $u_0 \in L^2(\T^d)$ and noise coefficient $\theta$ satisfies $\sum_{l} |l|^2 \theta_l^2 < + \infty$. By applying the compactness method \cite{Simon} and following the proof of \cite[Theorem 2.2]{FGL} and \cite[Theorem A.2]{LuoTang23}, we can obtain a subsequence of stochastic processes $\{u^{n_j} (\cdot) \}_{j}$ satisfying Lemma \ref{lem: uniform-bounded} that is tight in $C([0,t_0]; H^{-1}(\T^d))$ and converges in distribution to a solution $\Tilde{u}$ of equation \eqref{STE-Ito}. The pathwise uniqueness of equation \eqref{STE-Ito} (see \cite[Theorem 3]{Gal20}) implies the distribution uniqueness. Then for any fixed $R>0$, we have
    \begin{equation*}
    \E \Big[ \sup_{t \in [0,t_0]} \| u(t) \|_{H^{-1}}^2
    \wedge R\Big] = \lim_{j \rightarrow + \infty} \E \Big[ \sup_{t \in [0,t_0]} \| u^{n_j}(t) \|_{H^{-1}}^2
    \wedge R\Big] \leq 2 \|u_0\|_{H^{-1}}^2.
    \end{equation*}
    Let $R \rightarrow +\infty$, we get the desired result.
\end{proof}

Using the Markov property of solution $u$ to stochastic transport equation \eqref{STE-Ito}, we will prove a uniformly decaying rate of $\|u\|_{H^{-1}}^2$ over the time interval $[nt_0, (n+1)t_0]$ with the help of Theorem \ref{thm: mixing-exp} and Lemma \ref{lem: uniform-bounded}. The definition of $C(\theta,d)$ can be found in \eqref{def: C-theta-d}.

\begin{lemma} \label{lem: uniform-decay}
    If initial data $u_0 \in L^2(\T^d)$ and noise coefficient $\theta$ satisfies $ \| \theta \|_{h^1}^2= \sum_{l} |l|^2 \theta_l^2 < + \infty$, then the solution $u$ of equation \eqref{STE-Ito} has a uniformly exponential mixing:
\begin{itemize}
    \item[\rm (i)] When $d=2$ or $3$, there exists a constant $C$ that is only dependent on $d$ such that
    $$\E \Big[\sup_{t \in [n t_0,(n+1)t_0]} \|  u(t) \|_{H^{-1}}^2 \Big] \leq C \, e^{- \kappa C(\theta,d) \, n t_0/4} \|u_0\|_{L^2}^{2}, \quad \forall n \in \N.$$

    \item[\rm (ii)] When $d \geq 4$, there exists a constant $C$ that is only dependent on $d$ such that
    $$\E \Big[\sup_{t \in [n t_0,(n+1)t_0]} \| u(t) \|_{H^{-1}}^2 \Big] \leq  C \, e^{- 2 \kappa \frac{d-3}{d^2} \, C(\theta,d) \,  n t_0} \|u_0\|_{L^2}^{2}, \quad \forall n \in \N.$$
\end{itemize}
\end{lemma}

\begin{proof}
    By the Markov property of solution $u$ and the pathwise uniqueness of the solution to the stochastic transport equation \eqref{STE-Ito}, it holds
    \begin{equation*}
        u (t,x,\omega)|_{u_0(\cdot)}=u(t-nt_0,x,\vartheta_{nt_0} \omega)|_{u(nt_0,\cdot,\omega)} , \quad \forall t \in [n t_0,(n+1) t_0],
    \end{equation*}
    where $\vartheta: \R_{+} \times \Omega \rightarrow \Omega$ is the Wiener shift in $\Omega$ and $u (t,x,\omega)|_{u_0(\cdot)}$ is the solution to equation \eqref{STE-Ito} with the initial data $u_0$. Then, by Lemma \ref{lem: uniform-bounded}, we have
    \begin{align*}
        \E \Big[ \sup_{t \in [n t_0,(n+1)t_0]} \|  u(t,x,\omega) \|_{H^{-1}}^2 \Big] & = \E \Big[ \E \Big( \sup_{s \in [0,t_0]} \big\|  u(s,x,\vartheta_{nt_0} \omega)|_{u(nt_0,\cdot,\omega)} \big\|_{H^{-1}}^2 \big| \, \mathcal{F}_{n t_0} \Big) \Big] \\
        & \leq 2 \, \E  \|  u(n t_0) \|_{H^{-1}}^2 .
    \end{align*}
    To achieve the desired result, we use Theorem \ref{thm: mixing-exp} with $\beta=1$ for $d=2$ and $3$; for higher dimensions, we apply the same theorem with $\beta=1$ and $\epsilon=\frac{1}{d^2} < \frac{d-2}{d^2}$.
\end{proof}

Recalling $C(\theta,d)$ defined in \eqref{def: C-theta-d} and $D(\theta,d)$ defined in \eqref{def: D}, we can get the uniform decay estimate of $\E \big[\sup_{t \in [n t_0,(n+1)t_0]} \|  u(t) \|_{H^{-1}}^2 \big]$ from Lemma \ref{lem: uniform-decay} as follows:
\begin{equation*}
    \E \Big[ \sup_{t \in [n t_0,(n+1)t_0]} \| u(t) \|_{H^{-1}}^2 \Big] \leq  C \, e^{- D(\theta,d) \,  n t_0} \|u_0\|_{L^2}^{2}, \quad \forall n \in \N.
\end{equation*}
Using this estimate and Borel-Cantelli lemma, we will get the mixing properties in $\P$-a.s sense.

\begin{proof}[Proof of Theorem \ref{thm: mixing-a.s.}]

Let $t_0= \big(\frac{\sqrt{11}-3}{16} \big)^2 \frac{1}{ \pi^2 \, d \, \kappa \, \| \theta \|_{h^1}^2 }>0$ and constant $C$ be the same as in Lemma \ref{lem: uniform-decay}. We define the events
  $$A_n= \bigg\{\omega\in \Omega: \sup_{t\in [n t_0, (n+1)t_0]} \|u(t) \|_{H^{-1}}^2 > C \|u_0\|_{L^2}^2 \, e^{-\lambda n t_0 } \bigg\}, \quad n \geq 0.$$
Then, by Chebyshev's inequality and Lemma \ref{lem: uniform-bounded}, we obtain
\begin{align*}
    \sum_{n=0}^{+\infty} \P(A_n) & \leq \sum_{n=0}^{+\infty} \frac{e^{\lambda n t_0}}{C \|u_0\|_{L^2}^2 } \E \Big(\sup_{t\in [nt_0, (n+1)t_0]} \|u(t) \|_{H^{-1}}^2 \Big) \\
    &\leq \sum_{n=0}^{+\infty} e^{-(D(\theta,d)  \, \kappa-\lambda) n t_0 } <+\infty.
\end{align*}
By Borel-Cantelli lemma, for $\P$-a.s. $\omega\in \Omega$, there exists a big $N(\omega)\geq 0 $ such that
  \begin{equation*}
      \sup_{t\in [nt_0, (n+1)t_0]} \|u(t) \|_{H^{-1}}^2 \leq C \, e^{-\lambda n t_0 } \|u_0\|_{L^2}^2 \quad \mbox{for all } n \geq N(\omega).
  \end{equation*}
For $0 \leq n<N(\omega)$, we have
  \begin{equation*}
      \sup_{t\in [nt_0, (n+1)t_0]} \|u(t) \|_{H^{-1}}^2  \leq \sup_{t\in [nt_0, (n+1)t_0]} \|u(t) \|_{L^2}^2 \leq \|u_0\|_{L^2}^2 \leq e^{\lambda  (N(\omega) +1) t_0 } \, e^{-\lambda (n+1) t_0 } \|u_0\|_{L^2}^2.
  \end{equation*}
Therefore, defining the random constant $C_{\kappa,\theta,d}(\omega) := e^{\lambda t_0 }  \max \{C,e^{\lambda ( N(\omega)+1) t_0 } \}  $, for any $t\geq 0$, there exists a $n \in \mathbb N$ such that $t \in [n t_0,(n+1)t_0)$, thus
\begin{equation*}
    \|u(t) \|_{H^{-1}}^2  \leq \sup_{t \in [n t_0,(n+1)t_0]} \|u(t) \|_{H^{-1}}^2 \leq C_{\kappa,\theta,d}(\omega) \, e^{-\lambda t }   \|u_0\|_{L^2}^2.
\end{equation*}

It remains to show that $C_{\kappa,\theta,d}(\omega)$ has finite $q$-moment, for which it is enough to prove this for $e^{\lambda N(\omega) t_0 }$. To this end, we need to estimate the tail probability $\P(\{N(\omega) \geq k\})$.
Note that $N(\omega)$ may be defined as the largest integer $n$ such that
\begin{equation*}
    \sup_{t\in [n t_0, (n+1)t_0]} \|u(t) \|_{H^{-1}}^2 > C e^{-\lambda n t_0 }  \|u_0\|_{L^2}^2\, ;
\end{equation*}
hence $\{\omega\in \Omega: N(\omega) \geq k\} = \bigcup_{n=k}^\infty A_n$. Then for any $k \in \N$, we have
\begin{equation*}
    \P(\{N(\omega) \geq k\}) \leq \sum_{n=k}^\infty \P(A_n) \leq \sum_{n=k}^\infty  e^{-(D(\theta,d)  \, \kappa-\lambda) n t_0 } \leq \frac{e^{-(D(\theta,d)  \, \kappa-\lambda) t_0 \, k }}{(D(\theta,d)  \, \kappa-\lambda) t_0} .
\end{equation*}
As a result, for any $0<q<\frac{D(\theta,d)  \, \kappa}{\lambda}-1$,
\begin{align*}
    \E e^{q \lambda  N(\omega) t_0 } & = \sum_{k=0}^{+\infty} e^{q \lambda t_0 k } \, \P(\{N(\omega) = k\}) \\
    & \leq \sum_{k=0}^{+\infty} e^{q \lambda t_0 k } \frac{e^{-(D(\theta,d)  \, \kappa-\lambda) t_0 \, k }}{(D(\theta,d)  \, \kappa-\lambda) t_0} <+ \infty.
\end{align*}
So $C(\omega)$ has finite $q$-th moment.
\end{proof}

\subsection{Positive Sobolev norm estimate} \label{subsec: positive-norm}

In this subsection, we show the exponential growth of $\E \|u(t)\|^2_{H^\beta}$ ($0<\beta \leq 1$), which will lead to a lower bound of the mixing rate of solution $u$.

\begin{theorem} \label{thm: positive-sobolev}
    If $\supp \theta \subset \Z_0^d$ is compact and the initial data $u_0$ is smooth, then the solution $u$ of the stochastic transport equation \eqref{STE-Ito} satisfies
    \begin{equation} \label{eq: positive-sobolev-1}
        \E\| u(t) \|_{H^1}^2 = e^{8 \pi^2 \kappa  \| \theta \|_{h^1}^2 t} \| u_0 \|_{H^1}^2, \quad t \in [0,+\infty).
    \end{equation}
    If $ \| \theta \|_{h^1}^2=\sum_{k} |k|^2 \theta_k^2 < + \infty$ and $u_0 \in H^{\beta}(\T^d)$ for some $\beta \in (0,1]$, then
    \begin{equation} \label{eq: positive-sobolev-2}
        \E\| u(t) \|_{H^{\beta}}^2 \leq e^{8 \pi^2 \kappa \beta  \| \theta \|_{h^1}^2 t} \| u_0 \|_{H^{\beta}}^2 .
    \end{equation}
\end{theorem}
\begin{proof}
    Firstly, we prove \eqref{eq: positive-sobolev-1} and \eqref{eq: positive-sobolev-2} when $\supp \theta \subset \Z_0^d$ is compact and $u_0$ is smooth. Let $Z_k(t)= |2 \pi k|^2 \E |\hat{u}_k(t)|^2 $ and recall Proposition \ref{prop: Fourier coefficient}, we get
    \begin{equation} \label{eq: positive-norm}
        \sum_k Z_k(t) - \sum_k Z_k(0) =  8 \pi^2 C_d \kappa  \sum_{l,i} \sum_k \theta_{l}^2 (a_{l,i} \cdot k)^2 |k|^2 \int_0^t \left( \frac{Z_{k-l}(s)}{|k-l|^2}- \frac{Z_k(s)}{|k|^2} \right) \d s.
    \end{equation}
    For fixed $l \in \Z_0^d$, substituting $k$ for $\Tilde{k}+l$ in the above equation yields
    \begin{equation*}
        \sum_k Z_k(t) - \sum_k Z_k(0) =  8 \pi^2 C_d \kappa \sum_{l,i} \sum_{\Tilde{k}} \theta_{l}^2 (a_{l,i} \cdot \Tilde{k})^2 |\Tilde{k}+l|^2 \int_0^t \left( \frac{Z_{\Tilde{k}}(s)}{|\Tilde{k}|^2}- \frac{Z_{\Tilde{k}+l}(s)}{|\Tilde{k}+l|^2} \right) \d s.
    \end{equation*}
    Substituting $l$ for $\Tilde{l}=-l$ in \eqref{eq: positive-norm} and adding the above equation, we obtain
    \begin{align*}
         2 \Big( \sum_k Z_k(t) - \sum_k Z_k(0) \Big) & =  8 \pi^2 C_d \kappa \int_0^t \sum_{l,i} \sum_{k} \theta_{l}^2 (a_{l,i} \cdot (k+l))^2 \, \frac{|k|^2 -|k+l|^2}{|k+l|^2 } Z_{k+l}(s) \, \d s \\
         & \quad +  8 \pi^2 C_d \kappa \int_0^t \sum_{l,i} \sum_{k} \theta_{l}^2 (a_{l,i} \cdot k)^2 \,  \frac{|k+l|^2 -|k|^2}{|k|^2} Z_{k}(s) \, \d s
         .
    \end{align*}
    Substituting $k$ for $k-l$ in the first term of the right-hand side of the above identity, we get
    \begin{align*}
        \sum_k Z_k(t) - \sum_k Z_k(0) & = 4 \pi^2 C_d \kappa \int_0^t \sum_{l,i} \sum_{k} \theta_{l}^2 (a_{l,i} \cdot k)^2  \frac{|k+l|^2 + |k-l|^2 -2 |k|^2}{|k|^2 } Z_{k}(s) \, \d s \\
        & = 8 \pi^2 C_d \kappa \int_0^t \sum_{l,i} |l|^2 \theta_{l}^2 \sum_{k} \Big(a_{l,i} \cdot \frac{k}{|k|} \Big)^2  Z_{k}(s) \, \d s \\
        & = 8 \pi^2 \kappa \sum_{l} |l|^2 \theta_l^2 \int_0^t \sum_k Z_{k}(s) \, \d s,
    \end{align*}
    where the last step used Lemma \ref{lem: matrix-property}. Then, we obtain
    \begin{equation*}
        \E\| u(t) \|_{H^1}^2 = \sum_k Z_k(t) = e^{8 \pi^2 \kappa  \| \theta \|_{h^1}^2 t} \sum_k Z_k(0)= e^{8 \pi^2 \kappa  \| \theta \|_{h^1}^2 t} \| u_0 \|_{H^1}^2.
    \end{equation*}

    Now we prove inequality \eqref{eq: positive-sobolev-2} in the above setting. Let $(E_1, E_2)_{\theta, q}$ be the interpolation space of two Banach spaces $E_1$ and $E_2$ given by $K$-method (see \cite {Lunardi18}). Then the desired result can be derived by the interpolation theorem (see \cite[Chapter 1, Theorem 1.6]{Lunardi18}) and combining the following three facts: $\E\|u(t)\|_{L^2}^2\leq \|u_0\|_{L^2}^2$, $H^\beta(\T^d)=(H^1(\T^d), L^2(\T^d))_{\beta,2}$ and $L^2(\Omega; H^\beta(\T^d))=(L^2(\Omega; H^1(\T^d)), L^2(\Omega; L^2(\T^d)))_{\beta,2}$.

    Finally, we prove \eqref{eq: positive-sobolev-2} in the general case that $\sum_{k} |k|^2 \theta_k^2 < + \infty$ and $u_0 \in H^{\beta}(\T^d)$. In this case, we can find a sequence $\{u^n\}_{n}$ solving
    \begin{equation*}
        \d u^n = \kappa \Delta u^n \d t + \sqrt{C_d \kappa} \sum_{k,i} \theta_{k}^{(n)} e_{k} a_{k,i} \cdot \nabla u^n \d W_{t}^{k,i}, \quad u^n(0)=u_0^n,
    \end{equation*}
    where $\theta_{k}^{(n)}$ is defined as in Theorem \ref{thm: mixing-exp}, and the initial data $u_0^n$ converge to $u_0$ in $H^{\beta}(\T^d)$. Then $u^{n}$ satisfies the inequality \eqref{eq: positive-sobolev-2}, so there is a subsequence $\{u^{n_j}\}_j$ converging weakly to $\Tilde{u}$ in $L^2(\Omega; L^2([0,T]; H^{\beta}(\T^d)))$. Similar to the proof of \cite[Theorem 2]{Gal20}, we know $\Tilde{u}$ is just the unique solution $u$ to equation \eqref{STE-Ito}. Due to the collection of processes satisfying inequality \eqref{eq: positive-sobolev-2} in $L^2(\Omega; L^2([0,T]; H^{\beta}(\T^d)))$ is a closed convex set, so it is also weakly closed (see Lemma \ref{lem: convegence}), then the weak limit $u=\Tilde{u}$ also satisfies inequality \eqref{eq: positive-sobolev-2}.
\end{proof}

\begin{remark}
    If $\supp \theta \subset \Z_0^d$ is compact and the initial data $u_0$ is smooth, then for any $0<\beta<1$, the inequality $\|u\|_{L^2}^2 \leq \|u\|_{H^{\beta}} \|u\|_{H^{-\beta}} $ yields the lower bound of $\E \|u\|_{H^{-\beta}}^2$:
     \begin{equation*}
         \E \| u(t) \|_{H^{-\beta}}^2 \geq e^{- 8 \pi^2 \kappa \, \beta  \| \theta \|_{h^1}^2 t}  \, \frac{\|u_0\|_{L^2}^4}{\|u_0\|_{H^{\beta}}^2 }.
     \end{equation*}
\end{remark}

\section{Dissipation enhancement and Capi\'nski's conjecture} \label{sec: heat-equation}

In this section, we consider the stochastic heat equation \eqref{eq: stoch-heat} which is recalled here for reader's convenience:
\begin{equation}\label{eq: stoch-heat-1}
    \d u =\lambda u \, \d t+  \nu \Delta u \, \d t + \sqrt{C_d \kappa} \sum_{k,i} \theta_{k} e_{k} a_{k,i} \cdot \nabla u \circ \d W_{t}^{k,i},
\end{equation}
where $\lambda \geq 0$ and $\nu >0$. We aim to prove Theorem \ref{thm: heat}, which shows the dissipation enhancement and Capi\'nski's conjecture for the solution to \eqref{eq: stoch-heat-1}. Firstly, we prove the mixing properties of solutions to equation \eqref{eq: stoch-heat-1}, which is similar to Theorem \ref{thm: mixing-exp}.
\begin{lemma} \label{lem: heat-mixing}
    If initial data $u_0 \in L^2(\T^d)$ and $\theta \in \ell^2(\Z_0^d)$, then the solution $u$ to \eqref{eq: stoch-heat-1} satisfies
    \begin{equation*}
        \E \| u(t) \|_{H^{-1}}^2 \leq C \, e^{-(-2\lambda+8 \pi^2 \nu+D(\theta,d)  \, \kappa  )  \, t} \|u_0\|_{L^2}^{2},
    \end{equation*}
    where $D(\theta,d)$ is defined in \eqref{def: D} and $C$ is independent of noise coefficients $\kappa$ and $\theta$.
\end{lemma}
\begin{proof}
    Firstly, we assume that $\supp \theta \subset \Z_0^d$ is compact and the initial data $u_0$ is smooth. Let $u$ be the solution to equation \eqref{eq: stoch-heat-1}, then its Fourier coefficient $\hat{u}_k$ satisfies
    \begin{equation*}
    \frac{\d}{\d t} \E |\hat{u}_{k}|^2 = 2 \lambda \E |\hat{u}_{k}|^2  - 8 \pi^2 (\nu+\kappa) |k|^2 \E |\hat{u}_{k}|^2  + 8 \pi^2 C_d \kappa \sum_{l} \theta_{l}^2 |\Pi_l^\perp k|^2 \E |\hat{u}_{k-l}|^2 .
    \end{equation*}
    Let $Y_k(t)= \E |\hat{u}_{k}(t)|^2$. For any $p>1$, proceeding as in Lemma \ref{lem: sum-Fourier-ODE}, we get
    \begin{align*}
        \frac{\d}{\d t} \sum_{k} Y_{k}^p &= 2 \lambda \, p \sum_{k} Y_{k}^p - 8 \pi^2 \nu p \sum_{k} |k|^2  Y_{k}^p \\
        & \quad -4 \pi^2 C_{d} \kappa \, p  \sum_{l} \! \sum_{k} \theta_{l}^2 |\Pi_l^\perp k|^2 \big(Y_{k+l}-Y_{k} \big) \! \big(Y_{k+l}^{p-1}-Y_{k}^{p-1} \big) \\
        & \leq  (2\lambda- 8 \pi^2 \nu) p \sum_{k} Y_{k}^p  \!-4 \pi^2 C_{d} \kappa \, p  \sum_{l} \! \sum_{k} \theta_{l}^2 |\Pi_l^\perp k|^2 \big(Y_{k+l}-Y_{k} \big) \! \big(Y_{k+l}^{p-1}-Y_{k}^{p-1} \big).
    \end{align*}
    Similar as Theorem \ref{thm: Fourier-coefficient-exp-decay}, we obtain
    \begin{equation*}
       \|Y(t)\|_{\ell^p} \leq e^{- (-2\lambda+8 \pi^2 \nu+ C(\theta,d) \frac{p-1}{p^2} \kappa ) \, t} \|Y(0)\|_{\ell^p}.
    \end{equation*}
    For $d=2,3$, proceeding as in Theorem \ref{thm: mixing-exp} (ii), we have
    \begin{equation*}
        \E \| u(t) \|_{H^{-1}}^2 \leq C \, e^{-(-2\lambda+8 \pi^2 \nu+ \frac{1}{4} C(\theta,d)\kappa )  \, t} \|u_0\|_{L^2}^{2}.
    \end{equation*}
    For $d\geq 4$, proceeding as in Theorem \ref{thm: mixing-exp} (i) and choosing $\epsilon=\frac{1}{d^2}<\frac{d-2}{d}$, we get
    \begin{equation*}
        \E \| u(t) \|_{H^{-1}}^2 \leq C \, e^{-(-2\lambda+8 \pi^2 \nu+ 2 \frac{d-3}{d^2} C(\theta,d) \kappa) \, t} \|u_0\|_{L^2}^{2}.
    \end{equation*}
    Recalling the definition of $D(\theta,d)$ and combining the above estimates, we get the desired result when $u_0$ and $\theta$ are good enough. For general $u_0 \in L^2(\T^d)$ and $\theta \in \ell^2(\Z_0^d)$, following the approximation of Theorem \ref{thm: mixing-exp}, we get Lemma \ref{lem: heat-mixing} also holds for general case.
\end{proof}

Now, we will use the above mixing property to prove Theorem \ref{thm: heat}.

\begin{proof}[Proof of Theorem \ref{thm: heat}]
    By energy identity \eqref{eq: energy-identity-heat}, we have
    \begin{align*}
        - \frac{\d}{\d t} \frac{e^{2 \lambda t}}{ \E \|u(t)\|_{L^2}^2} &= - \frac{2 \lambda \, e^{2 \lambda t}}{ \E \|u(t)\|_{L^2}^2}+ \frac{e^{2 \lambda t}}{ \big(\E \|u(t)\|_{L^2}^2 \big)^2} \, \big( 2 \lambda \E \|u(t)\|_{L^2}^2 -2 \nu \E \|\nabla u(t)\|_{L^2}^2 \big) \\
        & =  - \frac{2 \nu e^{2 \lambda t}}{ \big(\E \|u(t)\|_{L^2}^2 \big)^2} \, \E \|u(t)\|_{H^{1}}^2 .
    \end{align*}
     The inequality $\E \|u(t)\|_{L^2}^2 \leq (\E \|u(t)\|_{H^{1}}^2)^{1/2}(\E \|u(t)\|_{H^{-1}}^2)^{1/2}$ and Lemma \ref{lem: heat-mixing} yield
    \begin{equation*}
        - \frac{\d}{\d t} \frac{e^{2 \lambda t}}{ \E \|u(t)\|_{L^2}^2} \leq  - \frac{2 \nu e^{2 \lambda t}}{ \E \|u(t)\|_{H^{-1}}^2} \leq - \frac{2 \nu}{C \|u_0\|_{L^2}^{2} } \, e^{(8 \pi^2 \nu+D(\theta,d)  \, \kappa  )  \, t}.
    \end{equation*}
    Then we have
    \begin{equation*}
         \frac{1}{ \|u_0\|_{L^2}^2} -  \frac{e^{2 \lambda t}}{ \E \|u(t)\|_{L^2}^2} \leq - \frac{2 \nu}{C (8 \pi^2 \nu+ D(\theta,d)  \, \kappa) } \big(e^{(8 \pi^2 \nu+D(\theta,d)  \, \kappa )  \, t}-1 \big) \frac{1}{\|u_0\|_{L^2}^2}.
    \end{equation*}
    Choosing $C\geq \frac{1}{4 \pi^2}$, we know $ \frac{2 \nu}{C (8 \pi^2 \nu+ D(\theta,d)  \, \kappa) } \leq 1$, then
    \begin{equation*}
        \E \|u(t)\|_{L^2}^2 \leq C \, \frac{8 \pi^2 \nu+ D(\theta,d)  \, \kappa}{2 \nu } \, e^{-(-2 \lambda + 8 \pi^2 \nu+D(\theta,d)  \, \kappa )  \, t} \|u_0\|_{L^2}^2.
    \end{equation*}
    So, we get the dissipation enhancement (or Capi\'nski's conjecture) in the average sense.

    For any $0<s<t$, we have the solution $u$ of \eqref{eq: stoch-heat-1} satisfies a.s.
    $$\|u(t)\|_{L^2}^2 \leq \exp \{ (2 \lambda-8 \pi^2 \nu) (t-s) \} \|u(s)\|_{L^2}^2.$$
    Combining the above energy estimates, for any $n \in \N$, we get
    \begin{align*}
       \E \Big[ \sup_{t\in [n, n+1]} \|u(t) \|_{L^2}^2 \Big] \leq (e^{ (2 \lambda-8 \pi^2 \nu) } \vee 1) \E \|u(n)\|_{L^2}^2 \leq \Tilde{C}  e^{-(-2 \lambda + 8 \pi^2 \nu+D(\theta,d)  \, \kappa )  \, n} \|u_0\|_{L^2}^2,
    \end{align*}
    where $\Tilde{C}:= C \, \frac{8 \pi^2 \nu+ D(\theta,d)  \, \kappa}{2 \nu } (e^{ (2 \lambda-8 \pi^2 \nu)} \vee 1)$. Define the events
    $$A_n= \bigg\{\omega\in \Omega: \sup_{t\in [n, n+1]} \|u(t) \|_{L^2}^2 > \Tilde{C} e^{-\gamma n  } \, \|u_0\|_{L^2}^2 \bigg\}, \quad n \geq 0.$$
Then, by Chebyshev's inequality and the fact $\gamma< -2 \lambda + 8 \pi^2 \nu+D(\theta,d)  \, \kappa$, we obtain
\begin{align*}
    \sum_{n=0}^{+\infty} \P(A_n) & \leq \sum_{n=0}^{+\infty} \frac{e^{\gamma n }}{\Tilde{C} \|u_0\|_{L^2}^2 } \E \Big(\sup_{t\in [n, n+1]} \|u(t) \|_{L^2}^2 \Big) \\
    &\leq \sum_{n=0}^{+\infty} e^{-(-\gamma-2 \lambda + 8 \pi^2 \nu+D(\theta,d)  \, \kappa)\, n } <+\infty.
\end{align*}
By Borel-Cantelli lemma, for $\P$-a.s. $\omega\in \Omega$, there exists a big $N(\omega)\geq 0 $ such that
  \begin{equation*}
      \sup_{t\in [n, n+1]} \|u(t) \|_{L^2}^2 \leq \Tilde{C} \, e^{-\lambda n t_0 } \|u_0\|_{L^2}^2 \quad \mbox{for all } n \geq N(\omega).
  \end{equation*}
For $0 \leq n<N(\omega)$, we have
  \begin{equation*}
      \sup_{t\in [n, n+1]} \|u(t) \|_{L^2}^2 \leq  (e^{ (2 \lambda-8 \pi^2 \nu) N(\omega) } \vee 1) \, e^{\gamma  (N(\omega) +1) } \, e^{-\gamma (n+1) } \|u_0\|_{L^2}^2.
  \end{equation*}
Let $C(\omega) := \max \{\Tilde{C},e^{\gamma  (N(\omega)+1) }, e^{(2 \lambda-8\pi^2 \nu+\gamma)  (N(\omega)+1) } \}  $. For any $t\geq 0$, there exists a $n \in \mathbb N$ such that $t \in [n,n+1)$, thus
\begin{equation*}
    \|u(t) \|_{L^2}^2  \leq \sup_{t \in [n,n+1]} \|u(t) \|_{L^2}^2 \leq C(\omega) \, e^{-\gamma t }   \|u_0\|_{L^2}^2.
\end{equation*}

It remains to show that $C(\omega)$ has finite moments, for which it is enough to show this for $e^{\gamma N(\omega)}$ and $e^{(2 \lambda-8\pi^2 \nu+\gamma)  N(\omega) }$. To this end, we need to estimate the tail probability $\P(\{N(\omega) \geq k\})$.
Note that $N(\omega)$ may be defined as the largest integer $n$ such that
\begin{equation*}
    \sup_{t\in [n, n+1]} \|u(t) \|_{H^{-1}}^2 > \Tilde{C} e^{-\lambda n }  \|u_0\|_{L^2}^2\, ;
\end{equation*}
hence $\{\omega\in \Omega: N(\omega) \geq k\} = \bigcup_{n=k}^\infty A_n$. Then for any $k \in \N$, we have
\begin{align*}
    \P(\{N(\omega) \geq k\}) & \leq \sum_{n=k}^\infty \P(A_n) \leq \sum_{n=k}^\infty  e^{-(-\gamma-2 \lambda + 8 \pi^2 \nu+D(\theta,d)  \, \kappa) \, n } \\
    & \leq \frac{e^{-(-\gamma-2 \lambda + 8 \pi^2 \nu+D(\theta,d)  \, \kappa) \, k }}{ 8 \pi^2 \nu+D(\theta,d)  \, \kappa -\gamma-2 \lambda } .
\end{align*}
As a result, for any $(q+1) \gamma < 8 \pi^2 \nu+D(\theta,d)  \, \kappa -2 \lambda$,
\begin{align*}
    \E \, e^{q \gamma  N(\omega) } & = \sum_{k=0}^{+\infty} e^{q \gamma k } \, \P(\{N(\omega) = k\}) \\
    & \leq \sum_{k=0}^{+\infty} e^{q \gamma k } \frac{e^{-(-\gamma-2 \lambda + 8 \pi^2 \nu+D(\theta,d)  \, \kappa) \, k }}{ 8 \pi^2 \nu+D(\theta,d)  \, \kappa -\gamma-2 \lambda } <+ \infty;
\end{align*}
and for any $(q+1) (2 \lambda-8\pi^2 \nu+\gamma) < D(\theta,d)  \, \kappa$,
\begin{align*}
    \E \,  e^{ q (2 \lambda-8\pi^2 \nu+\gamma) N(\omega) }  & = \sum_{k=0}^{+\infty} e^{q (2 \lambda-8\pi^2 \nu+\gamma) k } \, \P(\{N(\omega) = k\}) \\
    & \leq \sum_{k=0}^{+\infty} e^{ q(2 \lambda-8\pi^2 \nu+\gamma) k } \frac{e^{-(-\gamma-2 \lambda + 8 \pi^2 \nu+D(\theta,d)  \, \kappa) \, k }}{ 8 \pi^2 \nu+D(\theta,d)  \, \kappa -\gamma-2 \lambda } <+ \infty.
\end{align*}
So $C(\omega)$ has finite $q$-th moment for any $0<q<\min\{\frac{-2 \lambda + 8 \pi^2 \nu+D(\theta,d)  \, \kappa}{\gamma}, \frac{D(\theta,d)  \, \kappa}{\gamma+2 \lambda - 8 \pi^2 \nu}\}-1$.
\end{proof}

\section{Mixing for regularized stochastic 2D Euler equations} \label{sec: 2D Euler}

We consider the regularized stochastic 2D Euler equation \eqref{eq: 2D-Euler-approxiamate} with $\alpha>0$ on $\T^2$:
\begin{equation} \label{eq: 2D-Euler-approxiamate-1}
    \begin{cases}
        \d w^{(\alpha)} + u^{(\alpha)} \cdot \nabla w^{(\alpha)} \, \d t =  \sqrt{2 \kappa} \sum_{k} \theta_{k} \sigma_{k} \cdot \nabla w^{(\alpha)} \circ \d W_{t}^{k} , \\
        u^{(\alpha)} = \curl^{-1} (- \Delta/(4 \pi^2) )^{-\alpha/2} w^{(\alpha)},
    \end{cases}
\end{equation}
with initial data $w^{(\alpha)}_0 \in L^2(\T^2)$, where $\curl^{-1}$ is the well-known Biot-Savart operator. Recall that in the 2D case, for any $k\in \Z^2_0$, we write the only divergence free vector field simply as $\sigma_{k}= a_k e_k$. The first equation in \eqref{eq: 2D-Euler-approxiamate-1} is understood in the It\^o form:
  $$\d w^{(\alpha)} + u^{(\alpha)} \cdot \nabla w^{(\alpha)} \, \d t = \kappa\Delta w^{(\alpha)}\,\d t + \sqrt{2 \kappa} \sum_{k} \theta_{k} \sigma_{k} \cdot \nabla w^{(\alpha)} \, \d W_{t}^{k}. $$
Inspired by \cite{BFM11}, we will use Girsanov's theorem to deduce the mixing properties of solutions to \eqref{eq: 2D-Euler-approxiamate-1} from similar properties of equation \eqref{STE-Ito}.

Firstly, we give the meaning of weak solutions to \eqref{eq: 2D-Euler-approxiamate-1}. By a filtered probability space $(\Omega,\mathcal{F}_t,\P)$, we mean a probability space $(\Omega,\mathcal{F}_{\infty},\P)$ and a right-continuous filtration $\{\mathcal{F}_t\}_{t \geq 0}$ such that $\mathcal{F}_{\infty}$ is the $\sigma$-algebra generated by $\cup_{t \geq 0} \mathcal{F}_{t}$.

\begin{definition} \label{def: solution-2D-Euler-approxiamte}
    Given $w^{(\alpha)}_0 \in L^2(\T^2)$, a weak solution of equation \eqref{eq: 2D-Euler-approxiamate-1} consists of a filtered probability space $(\Omega,\mathcal{F}_t,\P)$, a sequence of standard complex Brownian motions $\{W_t^{k}\}_{k \in \Z_0^2}$ on $(\Omega,\mathcal{F}_t,\P)$ and an $L^2(\T^2)$-valued, $\mathcal{F}_t$-progressively measurable process $w^{(\alpha)}$ with weakly continuous paths, such that for every $\phi \in C^{\infty}(\T^2)$, $\P$-a.s. for all $t \geq 0$,
    \begin{align*}
        \< w^{(\alpha)}(t), \phi \> - \< w^{(\alpha)}_0, \phi \> & = \kappa \int_{0}^{t} \< w^{(\alpha)}(s), \, \Delta \phi \> \, \d s + \int_{0}^{t} \< w^{(\alpha)}(s) , u^{(\alpha)}(s) \cdot \nabla \phi \> \, \d s\\
        &\quad - \sqrt{2\kappa} \sum_{k} \theta_k \int_{0}^{t} \<  w^{(\alpha)}(s) , \sigma_{-k} \cdot \nabla \phi \> \, \d W_{s}^{k},
    \end{align*}
    where $u^{(\alpha)} = \curl^{-1} (- \Delta/(4 \pi^2) )^{-\alpha/2} w^{(\alpha)}$. Denote this solution by $\big(\Omega,\mathcal{F}_t,\P,w^{(\alpha)},\{W_t^{k}\}_{k \in \Z_0^2}\big)$, or simply by $w^{(\alpha)}$. We call $w^{(\alpha)}$ an energy controlled weak solution if for any fixed $T>0$,
    \begin{equation*}
       \P \mbox{-a.s.},\quad  \|w^{(\alpha)}(t)\|_{L^2} \leq \|w_0^{(\alpha)}\|_{L^2}, \quad \forall \, t \in [0,T].
    \end{equation*}
\end{definition}

Below we will always consider the special noise coefficient $\{\theta^{(\alpha)}_{k}\}_k$ defined in \eqref{def: special-noise}:
\begin{equation*}
    \theta^{(\alpha)}_{k} = \frac{1}{K_{\alpha}} \frac{1}{|k|^{1+\alpha}}\quad ( k \in \Z_0^2), \quad K_{\alpha}=\Big( \sum_{k} \frac{1}{|k|^{2+2\alpha}} \Big)^{1/2}.
\end{equation*}

Before proving the well-posedness and mixing properties of weak solutions to \eqref{eq: 2D-Euler-approxiamate-1} with the above noise coefficient $\theta=\{\theta^{(\alpha)}_{k} \}_k $, we first show the connection between the energy controlled weak solutions of \eqref{eq: 2D-Euler-approxiamate-1} and the solutions of \eqref{STE-Ito} by Girsanov's transform. Following the idea of \cite{BBF,BFM11,BFM10}, we will find a new probability measure $\Tilde{\P}$ on $(\Omega, \mathcal{F}_{\infty})$ such that
\begin{equation} \label{def: new Brownian motion}
    \Tilde{W}_{t}^{k} = W_{t}^{k} - \int_{0}^{t} \frac{ \<a_k \cdot u^{(\alpha)}(s) , e_k\>}{ \sqrt{2 \kappa} \, \theta_{k}^{(\alpha)}}   \, \d s \quad ( k \in \Z_0^2)
\end{equation}
are complex-valued Brownian motions under $\Tilde{\P}$. Hence, the nonlinear term of \eqref{eq: 2D-Euler-approxiamate-1} is absorbed into the new Brownian motions and equation \eqref{eq: 2D-Euler-approxiamate-1} becomes linear equation \eqref{STE-Ito} under $\Tilde{\P}$.

Assume that $\big(\Omega,\mathcal{F}_t,\P,w^{(\alpha)},\{W_t^{k}\}_{k \in \Z_0^2}\big)$ is an energy controlled solution of \eqref{eq: 2D-Euler-approxiamate-1}. Let $\{\hat{w}_k^{(\alpha)}\}_{k \in \Z_0^2}$ be the Fourier coefficients of vorticity $w^{(\alpha)}$, then the velocity is given by
\begin{equation*}
    u^{(\alpha)} = \curl^{-1} (- \Delta/(4 \pi^2) )^{-\alpha/2} w^{(\alpha)}= \frac{{\rm i}}{2 \pi} \sum_{k} \frac{k^{\perp}}{|k|^{2+\alpha}} \hat{w}_k^{(\alpha)} e_k.
\end{equation*}
Due to $w^{(\alpha)}$ is an energy controlled weak solution, the process
\begin{equation*}
    M_t := \sum_{k \in \Z_0^2} \int_{0}^{t} \frac{ \<a_k \cdot u^{(\alpha)}(s) , e_k\>}{ \sqrt{2 \kappa} \, \theta_{k}^{(\alpha)}}   \, \d W_s^{-k} = \frac{{\rm i}}{2 \pi } K_{\alpha} \sum_{k \in \Z_0^2} \frac{a_k \cdot k^{\perp}}{ \sqrt{2 \kappa} \, |k|} \int_{0}^{t} \hat{w}_k^{(\alpha)}(s) \, \d  W_s^{-k}
\end{equation*}
is a martingale with quadratic variation
\begin{align*}
    [M,M]_t & = \frac{K_{\alpha}^2}{4 \pi^2 \, \kappa } \, \sum_{k \in \Z_0^2}  \frac{(a_k \cdot k^{\perp})^2}{|k|^2}  \, \int_{0}^{t} |\hat{w}_k^{(\alpha)}(s)|^2 \, \d s \\
    & = \frac{K_{\alpha}^2}{4 \pi^2 \, \kappa } \int_{0}^{t} \|w^{(\alpha)}(s)\|_{L^2}^2 \, \d s  \leq \frac{K_{\alpha}^2}{4 \pi^2 \, \kappa } \|w^{(\alpha)}_0 \|_{L^2}^2 \, t.
\end{align*}
By Novikov's criterion, we see that $\mathcal{N}(M)_t:= \exp\big( M_t-\frac{1}{2} [M,M]_t \big)$ is an exponential martingale. We define the set function $\Tilde{\P}$ on $\cup_{t \geq 0} \mathcal{F}_{t}$ by setting
\begin{equation} \label{eq: expositional martingale}
    \frac{\d \Tilde{\P}}{\d \P} \Big|_{\mathcal{F}_t} =\mathcal{N}(M)_t= \exp \bigg( \sum_{k \in \Z_0^2} \int_{0}^{t} \frac{ \<a_k \cdot u^{(\alpha)}(s) , e_k\>}{ \sqrt{2 \kappa} \, \theta_{k}^{(\alpha)}}   \, \d W_s^{-k} - \frac{K_{\alpha}^2}{8 \pi^2 \, \kappa } \int_{0}^{t} \|w^{(\alpha)}(s)\|_{L^2}^2 \, \d s \bigg)
\end{equation}
for every $t \geq 0$, and extend this definition to the terminal $\sigma$-field $\mathcal{F}_{\infty}$. We cannot prove that $\Tilde{\P}$ is absolutely continuous with respect to $\P$ on $\mathcal{F}_{\infty}$, but the strict positivity of $\mathcal{N}(M)_t$ shows $\Tilde{\P}$ and $\P$ are equivalent on each $\mathcal{F}_{t}$. Applying Girsanov's theorem for both real and imaginary parts of $\{ \Tilde{W}_{t}^{k} \}_k$, we deduce that $\{ \Tilde{W}_{t}^{k} \}_k$ are standard complex Brownian motions defined on $(\Omega,\mathcal{F}_t,\Tilde{\P})$, satisfying \eqref{noise.1}.

Since $w^{(\alpha)}$ is a weak solution of \eqref{eq: 2D-Euler-approxiamate-1}, for every $\phi \in C^{\infty}(\T^2)$, $\P$-a.s. for all $t \geq 0$,
\begin{align*}
    \< w^{(\alpha)}(t), \phi \> - \< w^{(\alpha)}_0, \phi \> & = \kappa \int_{0}^{t} \< w^{(\alpha)}(s), \, \Delta \phi \> \, \d s + \int_{0}^{t} \< w^{(\alpha)}(s) , u^{(\alpha)}(s) \cdot \nabla \phi \> \, \d s\\
    &\quad - \sqrt{2\kappa} \sum_{k} \theta_k \int_{0}^{t} \<  w^{(\alpha)}(s) , \sigma_{-k} \cdot \nabla \phi \> \, \d W_{s}^{k}.
\end{align*}
Recalling the definition \eqref{def: new Brownian motion} and using the fact $u^{(\alpha)}= \sum_{k} \<a_k \cdot u^{(\alpha)}(s) , e_k\> \sigma_{k}$, we get
\begin{equation*}
    \< w^{(\alpha)}(t), \phi \> - \< w^{(\alpha)}_0, \phi \> = \kappa \int_{0}^{t} \< w^{(\alpha)}(s), \, \Delta \phi \> \, \d s - \sqrt{2\kappa} \sum_{k} \theta_k \int_{0}^{t} \<  w^{(\alpha)}(s) , \sigma_{-k} \cdot \nabla \phi \> \, \d \Tilde{W}_{s}^{k}.
\end{equation*}
Due to $\Tilde{\P}$ and $\P$ are equivalent on each $\mathcal{F}_{t}$, for any fixed $T>0$, we have
\begin{equation*}
   \Tilde{\P} \mbox{-a.s.}, \quad \|w^{(\alpha)}(t)\|_{L^2} \leq \|w_0^{(\alpha)}\|_{L^2}, \quad \forall \,  t \in [0,T].
\end{equation*}
So we transform the energy controlled weak solution $w^{(\alpha)}$ of \eqref{eq: 2D-Euler-approxiamate-1} to a solution of \eqref{STE-Ito} on the probability space $(\Omega, \mathcal{F}_t, \Tilde{\P})$. Notice that the above transformation process is reversible; similarly, we can transform a solution of \eqref{STE-Ito} to an energy controlled weak solution of \eqref{eq: 2D-Euler-approxiamate-1}.

Following the idea of \cite{BBF,BFM10} and the above discussions, we first prove the existence and uniqueness in the law of weak solutions to \eqref{eq: 2D-Euler-approxiamate-1} with the noise coefficient $\theta=\{\theta^{(\alpha)}_{k} \}_k $.

\begin{lemma} \label{lem: well-posedness 2D Euler-alpha}
    For any initial data $w_0^{(\alpha)} \in L^2(\T^2)$, the equation \eqref{eq: 2D-Euler-approxiamate-1} with the noise coefficient $\theta=\{\theta^{(\alpha)}_{k}  \}_k $ has an energy controlled weak solution and this solution is unique in law.
\end{lemma}

\begin{proof}
    As discussed above, we can similarly transform a solution of \eqref{STE-Ito} to an energy controlled weak solution of \eqref{eq: 2D-Euler-approxiamate-1}, so we get the existence of energy controlled weak solutions of \eqref{eq: 2D-Euler-approxiamate-1}. Next we turn to prove its uniqueness in law.

    Assume that $\big(\Omega^{(i)},\mathcal{F}_t^{(i)},\P^{(i)},w^{(\alpha,i)},\{ W_t^{(k,i)}\}_{k \in \Z_0^2}\big)$, $i=1,2$, are two energy controlled weak solutions of \eqref{eq: 2D-Euler-approxiamate-1} with the same initial data $w_0^{(\alpha)} \in L^2(\T^2)$. Define set functions $\Tilde{\P}^{(i)}$ by \eqref{eq: expositional martingale} with respect to $\big(\P^{(i)},w^{(\alpha,i)},\{ W_t^{(k,i)}\}_{k} \big)$ and the complex-valued Brownian motions $\{ \Tilde{W}_t^{(k,i)}\}_{k}$ by \eqref{def: new Brownian motion} on $\big(\Omega^{(i)}, \mathcal{F}_t^{(i)}, \Tilde{\P}^{(i)} \big)$ for each $i=1,2$. We define the martingale $\Tilde{M}_t^{(i)}$ under $\Tilde{\P}^{(i)}$ by
    \begin{equation*}
        \Tilde{M}_t^{(i)} := - \frac{{\rm i}}{2 \pi} K_{\alpha} \sum_{k \in \Z_0^2} \frac{a_k \cdot k^{\perp}}{ \sqrt{2 \kappa} \, |k|} \int_{0}^{t} \hat{w}_k^{(\alpha,i)}(s) \, \d  \Tilde{W}_s^{(-k,i)},
    \end{equation*}
    then $(\d \P^{(i)} / \d \Tilde{\P}^{(i)}) \big|_{\mathcal{F}_t} =\exp \big(\Tilde{M}_t^{(i)} -\frac{1}{2} \big[ \Tilde{M}^{(i)},\Tilde{M}^{(i)} \big]_t \big)$. As mentioned earlier in this subsection, $w^{(\alpha,i)}$ is a solution of \eqref{STE-Ito} under $\Tilde{\P}^{(i)}$. Applying the pathwise uniqueness of \eqref{STE-Ito}, as proved in \cite[Theorem 3]{Gal20}, we conclude that $(\Tilde{M}^{(1)}, w^{(\alpha,1)})$ and $(\Tilde{M}^{(2)}, w^{(\alpha,2)})$ have the same law.

    For any fixed $T>0$, given $n \geq 1$, $t_1,t_2,\cdots,t_n \in [0,T]$ and a bounded measurable function $F: \big(L^2(\T^2)\big)^n \rightarrow \R$, we obtain that for each $i=1,2$,
    \begin{align*}
        & \E^{\P^{(i)}} \Big[ F \big(w^{(\alpha,i)}(t_1),w^{(\alpha,i)}(t_2),\cdots, w^{(\alpha,i)}(t_n) \big) \Big] \\
        & \quad = \E^{\Tilde{\P}^{(i)}} \Big[ \exp \Big(\Tilde{M}_T^{(i)} -\frac{1}{2} \big[\Tilde{M}^{(i)}, \Tilde{M}^{(i)} \big]_T \Big) \, F \big(w^{(\alpha,i)}(t_1),w^{(\alpha,i)}(t_2),\cdots, w^{(\alpha,i)}(t_n) \big) \Big],
    \end{align*}
    where $\E^{\P^{(i)}}$ and $\E^{\Tilde{\P}^{(i)}}$ are the expectations under $\P^{(i)}$ and $\Tilde{\P}^{(i)}$, respectively. Using the uniqueness of the law of $(\Tilde{M}^{(i)}, w^{(\alpha,i)})$ with respect to $\Tilde{\P}^{(i)}$, we have
    \begin{equation*}
        \E^{\P^{(1)}} \! \Big[ F \big(w^{(\alpha,1)}(t_1),w^{(\alpha,1)}(t_2),\cdots, w^{(\alpha,1)}(t_n) \big) \Big]= \E^{\P^{(2)}} \! \Big[ F \big(w^{(\alpha,2)}(t_1),w^{(\alpha,2)}(t_2),\cdots, w^{(\alpha,2)}(t_n) \big) \Big].
    \end{equation*}
    Then we obtain the uniqueness in law of $w^{(\alpha,i)}$ with respect to $\P^{(i)}$ on $C([0,T]; L_w^2(\T^2))$.
\end{proof}

Now we will use Girsanov's transform and the mixing properties of solutions of \eqref{STE-Ito} to prove that the solution to \eqref{eq: 2D-Euler-approxiamate-1} is exponentially mixing in the averaged sense.

\begin{proof}[Proof of Theorem \ref{thm: mixing-Euler}]

    Assume that $\big(\Omega,\mathcal{F}_t,\P,w^{(\alpha)},\{W_t^{k}\}_{k \in \Z_0^2}\big)$ is an energy controlled weak solution of \eqref{eq: 2D-Euler-approxiamate-1}. Lemma \ref{lem: well-posedness 2D Euler-alpha} gives the uniqueness of the law of $w^{(\alpha)}$, so we only need to prove this weak solution has the mixing property.

    Recall the discussions in the beginning of this subsection.
    Define set function $\Tilde{\P}$ by \eqref{eq: expositional martingale} and the complex-valued Brownian motions $\{ \Tilde{W}_t^{k}\}_{k}$ by \eqref{def: new Brownian motion}. Then $\big(\Omega,\mathcal{F}_t,\Tilde{\P},w^{(\alpha)},\{\Tilde{W}_t^{k}\}_{k \in \Z_0^2}\big)$ is a solution to stochastic transport equation \eqref{STE-Ito}. Applying Theorem \ref{thm: mixing-exp}, we obtain
    \begin{equation*}
        \E^{\Tilde{\P}}\|w^{(\alpha)}(t) \|_{H^{-1}}^2 \leq C_0 \, e^{- \frac{\pi^2}{4} \kappa  \, \| \theta \|_{h^{-1}}^2 t } \, \|w^{(\alpha)}_0\|_{L^2}^{2}, \quad \forall\, t \geq 0,
    \end{equation*}
    where $\E^{\Tilde{\P}}$ is the expectation under $\Tilde{\P}$. Using H\"older's inequality and \eqref{eq: expositional martingale}, we have
    \begin{align*}
        \E \| w^{(\alpha)}(t) \|_{H^{-1}}^2 & = \E^{\Tilde{\P}} \bigg[ \| w^{(\alpha)}(t) \|_{H^{-1}}^2\, \frac{\d \P}{\d \Tilde{\P}} \Big|_{\mathcal{F}_t} \bigg] \leq \Big[ \E^{\Tilde{\P}} \| w^{(\alpha)}(t) \|_{H^{-1}}^4 \Big]^{1/2} \, \bigg[ \E^{\Tilde{\P}} \Big( \frac{\d \P}{\d \Tilde{\P}} \Big|_{\mathcal{F}_t} \Big)^2 \bigg]^{1/2} \\
        & \leq \esssup_{(\Omega, \Tilde{\P})} \| w^{(\alpha)}(t) \|_{H^{-1}} \, \big( \E^{\Tilde{\P}}\|w^{(\alpha)}(t) \|_{H^{-1}}^2 \big)^{1/2} \, \Big[ \E \exp \Big(-M_t + \frac{1}{2} [M,M]_t \Big) \Big]^{1/2}.
    \end{align*}
    Due to $w^{(\alpha)}$ is an energy controlled solution and the equivalence of $\Tilde{\P}$ and $\P$ on $\mathcal{F}_{t}$, it holds
    \begin{equation*}
        \Tilde{\P} \mbox{-a.s.}, \quad \|w^{(\alpha)}(t) \|_{H^{-1}}^2 \leq \frac{1}{4 \pi^2} \|w^{(\alpha)}(t) \|_{L^2}^2 \leq \frac{1}{4 \pi^2}\|w_0^{(\alpha)} \|_{L^2}^2.
    \end{equation*}
    Recall that $\|w_0^{(\alpha)} \|_{L^2}^2 \leq R$, then $\|w^{(\alpha)}(s) \|_{L^2}^2 \leq \|w_0^{(\alpha)} \|_{L^2}^2 \leq R $ holds $\P$-a.s. for all $s \in [0,t]$. Since $\exp\big(-M_t -\frac{1}{2} [M,M]_t \big)$ is an exponential martingale under $\P$, we obtain
    \begin{align*}
        \E \exp \Big(-M_t + \frac{1}{2} [M,M]_t \Big) & \leq \esssup_{(\Omega,\P)} \, \exp \big( [M,M]_{t} \big) \times \E \exp \Big(-M_t - \frac{1}{2} [M,M]_t \Big)  \\
        & = \esssup_{(\Omega,\P)} \, \exp \Big(\frac{K_{\alpha}^2}{4 \pi^2 \kappa } \int_{0}^{t} \|w^{(\alpha)}(s) \|_{L^2}^2\, \d s\Big) \leq \exp \Big(\frac{ K_{\alpha}^2 \, R }{4 \pi^2 \kappa } \, t \Big).
    \end{align*}
    Combining the above estimates, we have
    \begin{equation*}
        \E \| w^{(\alpha)}(t) \|_{H^{-1}}^2 \leq \frac{\sqrt{C_0}}{2 \pi} \, \exp \Big\{ \Big( \frac{ K_{\alpha}^2 \, R }{8 \pi^2 \kappa }- \frac{\pi^2}{8} \kappa  \, \| \theta \|_{h^{-1}}^2 \Big) t \Big\} \, \|w^{(\alpha)}_0\|_{L^2}^{2}.
    \end{equation*}
    For any fixed $\lambda>0$, we choose $\kappa$ big enough such that
    \begin{equation*}
        \frac{\pi^2}{8} \kappa  \, \| \theta \|_{h^{-1}}^2 = \frac{\pi^2 \kappa }{8} \, \frac{1}{K_{\alpha}^2} \sum_{l} \frac{1}{|l|^{4+2\alpha}}  \geq \frac{ K_{\alpha}^2 \, R}{8 \pi^2 \kappa } +\lambda.
    \end{equation*}
    Inserting it into the estimate of $\E \|w^{(\alpha)}(t) \|_{H^{-1}}^2$, we get the desired result.
\end{proof}

\begin{remark}
    Due to $u^{(\alpha)} = \curl^{-1} (-\Delta/(4\pi^2))^{-\alpha/2} w^{(\alpha)}$, we have
    \begin{equation*}
    \E \|u^{(\alpha)}(t) \|_{H^{\alpha}}^2 \leq C \, e^{-\lambda t} \, \|w_0^{(\alpha)}\|_{L^2}^2, \quad \forall \, t \ge 0.
    \end{equation*}
\end{remark}

\begin{remark}
The above arguments may not work for the stochastic 2D Euler equation with transport noise, namely, \eqref{eq: 2D-Euler-approxiamate-1} with $\alpha=0$. To relate the stochastic 2D Euler equation with the linear equation \eqref{STE-Ito} by Girsanov's transform, a sufficient condition is
\begin{equation*}
    \P \mbox{-a.s.}, \quad \sum_{k \in \Z_0^2} \int_{0}^{t}  \frac{|\hat{w}_k^{(\alpha)}(s)|^2}{|k|^2 \theta_k^2}  \, \d s  < +\infty.
\end{equation*}
But the existence of weak solutions with such regularity is out of reach for the moment, see e.g. Theorem \ref{thm: positive-sobolev}; further discussions can be found in \cite[Section 6]{GalLuo23}.
\end{remark}

\appendix

\section{The proof of Theorem \ref{thm: Fourier-coefficient-exp-decay} ({$d \geq 3$})} \label{appendix: proof}

\begin{proof}[Proof of Theorem \ref{thm: Fourier-coefficient-exp-decay} ($d \geq 3$)]

Recall the notations in Section \ref{subsec: setting}. Similar to the proof of Theorem \ref{thm: Fourier-coefficient-exp-decay} ($d =2$), we need the following steps to prove Theorem \ref{thm: Fourier-coefficient-exp-decay} when $d \geq 3$.

\textbf{Step 1: constructing orbits.} For any fixed $l_1 \in \mathbb{Z}_{0}^d$, there exists $l_2 \in \mathbb{Z}_{0}^d$ such that $l_1, l_2$ are linearly independent and $|l_1|=|l_2|$. For any $l_2 \in \mathbb{Z}_0^d$ with $|l_2| = |l_1|$, it is linearly independent with $l_1$ if and only if $l_2 \neq \pm l_1$. For the two linearly independent vectors $l_1,l_2 \in \Z_0^d$, we can decompose $\Z_0^d$ as $\cup_{h \in P_Z^{\perp}(l_1,l_2)} \S_h(l_1,l_2)$. For fixed $h \in P_Z^{\perp}(l_1,l_2)$, we define the quadrant $Q_{h,1}:=\mathring{Q}_{h,1} \cup \partial Q_{h,1,1} \cup \partial Q_{h,1,2} $ as follows:
\begin{align*}
    \mathring{Q}_{h,1} & =\{z=a(z) \, l_1+ b(z) \, l_2+ h \in \Z_{0}^d ; \; a(z) , \, b(z) > 0\}, \\
    \partial Q_{h,1,1} & = \{z=b(z) \, l_2+ h \in \Z_{0}^d ; \; b(z) > 0\}, \\
    \partial Q_{h,1,2} & = \{z=a(z) \, l_1+ h \in \Z_{0}^d ; \;  a(z) > 0\},
\end{align*}
where for $z\in \mathring{Q}_{h,1}$, $a(z)$ and $b(z)$ are determined by
  $$\begin{pmatrix} a(z) \\ b(z) \end{pmatrix} = \begin{pmatrix} |l_1|^2 &\, l_1\cdot l_2 \\ l_1\cdot l_2 &\, |l_2|^2 \end{pmatrix}^{-1} \! \begin{pmatrix} z\cdot l_1 \\ z\cdot l_2 \end{pmatrix}= \frac{1}{|l_1|^2|l_2|^2- (l_1 \cdot l_2)^2} \begin{pmatrix} |l_2|^2 &\, -l_1 \cdot l_2 \\ -l_1 \cdot l_2 &\, |l_1|^2 \end{pmatrix}
  \begin{pmatrix} z\cdot l_1 \\ z\cdot l_2 \end{pmatrix} . $$
For $i=2,3,4$, the other quadrants $Q_{h,i}$ can be defined similarly. Note that if $h \in \S_{h}(l_1, l_2) \cap \Z_0^d$, then $ \cup_{i=1}^4 Q_{h,i}$ is equal to $\S_{h}(l_1, l_2) \setminus \{ h \}$. Thus, we need some additional discussions in Step 4.

Now, we define two classes of orbits to cover $\mathring{Q}_{h,1}$. For any $z \in \mathring{Q}_{h,1} \cup \partial Q_{h,1,1}$, we define the first class of orbits starting from $z$ as
\begin{equation*}
     \Gamma_{h,1,1}(z,n) = z + \Big\lfloor \frac{n+1}{2} \Big\rfloor \, l_1 + \Big\lfloor \frac{n}{2} \Big\rfloor \, l_2,  \quad n \in \N,
\end{equation*}
where $\lfloor x \rfloor$ is the largest integer that is less than or equal to $x\in \R$. Similarly, for any $z \in \mathring{Q}_{h,1} \cup \partial Q_{h,1,2}$, we define the second class of orbits starting from $z$ as
\begin{equation*}
     \Gamma_{h,1,2}(z,n) = z + \Big\lfloor \frac{n}{2} \Big\rfloor \, l_1 + \Big\lfloor \frac{n+1}{2} \Big\rfloor \, l_2,  \quad n \in \N.
\end{equation*}
The orbit $\Gamma_{h,i,1}$ and $\Gamma_{h,i,2}$ on $Q_{h,i}$ can be defined similarly to the orbits on $Q_{h,1}$.

\textbf{Step 2: estimates of orbits.} Let $\O_{h,1} = \{\Gamma_{h,1,1}(z,\cdot); z \in \mathring{Q}_{h,1} \cup \partial Q_{h,1,1} \} \cup \{\Gamma_{h,1,2}(z,\cdot); z \in \mathring{Q}_{h,1} \cup \partial Q_{h,1,2} \}$ be the set of orbits on $Q_{h,1}$. We assert that the orbit $O(\cdot) \in \mathcal{O}_{h,1}$ satisfies
\begin{equation} \label{eq: estimate-orbit}
    \sum_{n \in \N} |\Pi^{\perp}_\Delta O(n)|^2 \big(Y_{O(n+1)}^{p-1}-Y_{O(n)}^{p-1}\big) \big(Y_{O(n+1)}-Y_{O(n)} \big) \geq \frac{p-1}{8 p^2 |l_1|^2 } \sum_{n \in \N} Y_{O(n)}^{p}
\end{equation}
for any $p>1$, where $\Pi^{\perp}_{\Delta} O(n)= \Pi^{\perp}_{\Delta O(n)} O(n)$ with $\Delta O(n)= O(n+1) -O(n)$. By Lemma \ref{lem: sequence sum}, to prove the estimate \eqref{eq: estimate-orbit}, we only need to prove
\begin{equation} \label{eq: estimate-orbit-project}
    |\Pi^{\perp}_{\Delta} O(n)|^2= |\Pi^{\perp}_{O(n+1)-O(n)} O(n)|^2 \geq \frac{(n+1)^2}{4 |l_1|^2}, \quad \forall \, n \in  \N.
\end{equation}
   By the construction of the orbits, we know that $\Delta O(n)$ equals $l_1$ or $l_2$. We start with proving the inequality \eqref{eq: estimate-orbit-project} for the first class of orbits $\Gamma_{h,1,1}(z,\cdot)$. When $n$ is even, we have
\begin{equation*}
    \big|\Pi^{\perp}_{\Delta} O(n) \big|^2 = \big| \Pi^{\perp}_{l_1} \Gamma_{h,1,1}(z,n) \big|^2 =\big|\Pi^{\perp}_{l_1} z \big|^2+n \, \Pi^{\perp}_{l_1} z \cdot  \Pi^{\perp}_{l_1} l_2 + \frac{n^2}{4} \big|\Pi^{\perp}_{l_1} l_2\big|^2 .
\end{equation*}
Due to $z=a(z) \, l_1+ b(z) \, l_2 +h \in \mathring{Q}_{h,1} \cup \partial Q_{h,1,1}$, we have $b(z)>0$ and
\begin{equation*}
    \Pi^{\perp}_{l_1} z \cdot  \Pi^{\perp}_{l_1} l_2   = b(z) \, \big| \Pi^{\perp}_{l_1} l_2\big|^2+ \Pi^{\perp}_{l_1} h \cdot \Pi^{\perp}_{l_1} l_2
    = b(z) \, \big| \Pi^{\perp}_{l_1} l_2 \big|^2 > 0,
\end{equation*}
where the last step follows from the fact $\Pi^{\perp}_{l_1} h \cdot \Pi^{\perp}_{l_1} l_2  = h \cdot \Pi^{\perp}_{l_1} l_2 =0$. Due to $z,l_1$, and $l_2$ belong to $\Z_0^d$, the positivity of $\Pi^{\perp}_{l_1} z \cdot  \Pi^{\perp}_{l_1} l_2$ implies
\begin{equation*}
   \Pi^{\perp}_{l_1} z \cdot \Pi^{\perp}_{l_1} l_2= \Big( z - \frac{z \cdot l_1}{|l_1|^2} l_1\Big) \cdot \Big( l_2 - \frac{l_2 \cdot l_1}{|l_1|^2} l_1\Big)  = \frac{(z \cdot l_2)|l_1|^2- (z \cdot l_1) (l_2 \cdot l_1)}{|l_1|^2} \geq \frac{1}{|l_1|^2}.
\end{equation*}
Similarly, we have $|\Pi^{\perp}_{l_1} z|^2 \geq  \frac{1}{|l_1|^2}$ and $|\Pi^{\perp}_{l_1} l_2|^2 \geq \frac{1}{|l_1|^2}$. Then, when $n$ is even,
\begin{equation*}
    |\Pi^{\perp}_{\Delta} O(n)|^2 \geq \frac{1}{|l_1|^2} \Big(1+n+\frac{n^2}{4} \Big) \geq \frac{(n+1)^2}{4|l_1|^2 }.
\end{equation*}
For odd $n\in \N$, similar to the estimates when $n$ is even, we have
\begin{align}
    |\Pi^{\perp}_{\Delta} O(n)|^2 & = \big|\Pi^{\perp}_{l_2} z\big|^2 + (n+1) \, \Pi^{\perp}_{l_2} z \cdot \Pi^{\perp}_{l_2} l_1 + \frac{(n+1)^2}{4} \big| \Pi^{\perp}_{l_2} l_1\big|^2 \nonumber \\
    & \geq \frac{(n+1)^2}{4} \big| \Pi^{\perp}_{l_2} l_1\big|^2 \geq  \frac{(n+1)^2}{4|l_2|^2 }. \label{eq: odd-estimate}
\end{align}
By the same arguments as in the proof of $d=2$, Lemma \ref{lem: sequence sum} and the above estimates yield that the inequality \eqref{eq: estimate-orbit} holds for all orbits on the quadrant $Q_{h,i}$, where $i=1,2,3,4$.

\textbf{Step 3: covering $\S_{h}(l_1,l_2)\setminus \{h\}$ exactly twice.} Let $I_{h,1,1}$ and $I_{h,1,2}$ be
\begin{align*}
       I_{h,1,1} & = \{z \in \mathring{Q}_{h,1} ; \, z-l_2 \notin Q_{h,1} \} \cup \partial Q_{h,1,1}, \\
       I_{h,1,2} & = \{z \in \mathring{Q}_{h,1} ; \, z-l_1 \notin Q_{h,1} \} \cup \partial Q_{h,1,2}.
\end{align*}
Let $\Bar{\O}_{h,1} = \{\Gamma_{h,1,1}(z,\cdot); z \in I_{h,1,1} \} \cup \{\Gamma_{h,1,2}(z,\cdot); z \in I_{h,1,2} \}$ be the set of all the orbits starting from $I_{h,1,1}$ and $I_{h,1,2}$. Similarly, we can define the orbits space $\Bar{\O}_{h,i}$ for each $i=2,3,4$. By the same arguments as in the proof of $d=2$, we know that the orbits in $\cup_{i=1}^4 \Bar{\O}_{h,i}$ cover $\S_{h}(l_1,l_2)\setminus \{h\}$ exactly twice.

\textbf{Step 4: the additional step.} As mentioned in Step 1, when $h \in \S_{h}(l_1, l_2) \cap \Z_0^d$, we need an additional step. To take $h$ into account, we define the new orbits as follows:
\begin{align*}
    O_{h,1}(n) := h + \Big\lfloor \frac{n}{2} \Big\rfloor \, l_1 + \Big\lfloor \frac{n+1}{2} \Big\rfloor \, l_2,  \quad \forall n \in \N; \\
    O_{h,2}(n) := h + \Big\lfloor \frac{n+1}{2} \Big\rfloor \, l_1 + \Big\lfloor \frac{n}{2} \Big\rfloor \, l_2,  \quad \forall n \in \N.
\end{align*}
Then $O_{h,1}(1+\cdot)$ and $ O_{h,2}(1+\cdot)$ belong to $\Bar{\O}_{h,1}$. We define
\begin{equation*}
    \Bar{\O}_{h,1}^{\ast}:= \{ O_{h,1}(\cdot), \, O_{h,2}(\cdot) \} \cup \Bar{\O}_{h,1} \setminus \{ O_{h,1}(1+\cdot), \, O_{h,2}(1+\cdot) \}.
\end{equation*}
Then $\cup_{i=2}^4 \Bar{\O}_{h,i} \cup \Bar{\O}_{h,1}^{\ast}$ cover $\S_{h}(l_1,l_2)$ exactly twice. We need to establish a new estimate for $O_{h,1}$ and $O_{h,2}$ like \eqref{eq: estimate-orbit}. For orbit $O_{h,1}$, when $n$ is even, we have
\begin{align*}
    |\Pi^{\perp}_{\Delta} O_{h,1}(n)|^2 =|h|^2+ \frac{n^2}{4} \big|\Pi^{\perp}_{l_1} l_2\big|^2 \geq  \frac{1}{|l_1|^2} \Big( 1+\frac{n^2}{4} \Big) \geq \frac{(n+1)^2}{5|l_1|^2 }.
\end{align*}
When $n$ is odd, the estimate is similar to \eqref{eq: odd-estimate}. So for any $n \in \N$, we have $|\Pi^{\perp}_{\Delta} O_{h,1}(n)|^2 \geq \frac{(n+1)^2}{5 |l_1|^2}$ and similarly $|\Pi^{\perp}_{\Delta} O_{h,2}(n)|^2 \geq \frac{(n+1)^2}{5 |l_1|^2}$. By Lemma \ref{lem: sequence sum}, for $i=1,2$, we have
\begin{equation*}
    \sum_{n \in \N} |\Pi^{\perp}_\Delta O_{h,i}(n)|^2 \big(Y_{O_{h,i}(n+1)}^{p-1}-Y_{O_{h,i}(n)}^{p-1}\big) \big(Y_{O_{h,i}(n+1)}-Y_{O_{h,i}(n)} \big) \geq \frac{p-1}{10 p^2 |l_1|^2 } \sum_{n \in \N} Y_{O_{h,i}(n)}^{p}.
\end{equation*}

\textbf{Step 5: the final step.} By the same arguments as in the proof of $d = 2$, we obtain
\begin{align}
        & \sum_{k } \, \sum_{l=l_1,l_2,-l_1,-l_2 } \theta_{l}^2 \, \big|\Pi^{\perp}_{l} k \big|^2  \big(Y_{k+l}^{p-1}-Y_{k}^{p-1} \big) \big(Y_{k+l}-Y_{k} \big) \nonumber \\
       & \quad = \sum_{h \in P_{Z}^{\perp}(l_1,l_2)} \sum_{k \in \S_{h}(l_1,l_2) } \, \sum_{l=l_1,l_2,-l_1,-l_2 } \theta_{l}^2 \, \big|\Pi^{\perp}_{l} k \big|^2  \big(Y_{k+l}^{p-1}-Y_{k}^{p-1} \big) \big(Y_{k+l}-Y_{k} \big) \nonumber \\
       & \quad \geq \frac{2(p-1)}{5 p^2 } \frac{\theta_{l_1}^2}{|l_1|^2} \sum_{h \in P_{Z}^{\perp}(l_1,l_2)} \sum_{k \in \S_h(l_1,l_2)} Y_{k}^{p}= \frac{2(p-1)}{5 p^2 } \frac{\theta_{l_1}^2}{|l_1|^2} \sum_{k } Y_{k}^{p}. \label{eq: estimate-for-gronwall-1}
\end{align}

For any fixed $l_0 \in \Z_0^d$, there are many choices of $l_1$ and $l_2$ such that $|l_1|=|l_2|=|l_0|$ and $l_1,l_2$ are linearly independent. We define the set of all the pairs $(l_1,l_2)$ as
\begin{equation*}
    L_{|l_0|}:= \big\{ (l_1,l_2) \in \Z_0^d \otimes \Z_0^d; \, \mbox{$l_1$ and $l_2$ are linearly independent and } |l_1|=|l_2|=|l_0| \big\}.
\end{equation*}
Let $N_{|l_0|}=\# \{l \in \Z_0^d ; \, |l|=|l_0| \}$. By the discussion in Step 1, we get $\# L_{|l_0|}= N_{|l_0|} (N_{|l_0|}-2) $.

For any fixed $l_0 \in \Z_0^d$, the inequality \eqref{eq: estimate-for-gronwall-1} gives
\begin{align*}
    & \sum_{k} \, \sum_{ \substack{l \in \Z_0^d \\ |l|=|l_0|} } \theta_{l}^2 \, \big|\Pi^{\perp}_{l} k \big|^2  \big(Y_{k+l}^{p-1}-Y_{k}^{p-1} \big) \big(Y_{k+l}-Y_{k} \big) \\
    &= \frac{1}{4 \,  (N_{|l_0|}-2 )}  \sum_{k} \, \sum_{(l_1,l_2) \in L_{|l_0|}} \sum_{l=l_1,l_2,-l_1,-l_2} \theta_{l}^2 \, \big|\Pi^{\perp}_{l} k \big|^2  \big(Y_{k+l}^{p-1}-Y_{k}^{p-1} \big) \big(Y_{k+l}-Y_{k} \big) \\
    & \geq \frac{1}{4 \, (N_{|l_0|}-2 )} \sum_{(l_1,l_2) \in L_{|l_0|}}  \frac{2(p-1)}{5 p^2 }  \frac{\theta_{l_1}^2}{|l_1|^2}  \sum_{k } Y_{k}^{p} = \frac{p-1}{10 p^2 } \sum_{ \substack{l \in \Z_0^d \\ |l|=|l_0|} } \frac{\theta_{l}^2}{|l|^2} \sum_{k } Y_{k}^{p}.
\end{align*}
Then by Lemma \ref{lem: sum-Fourier-ODE} and Gr\"onwall inequality, we get inequality \eqref{eq: Fourier-coefficient-exp-decay}.
\end{proof}

\section*{Acknowledgements}

The first author is grateful to the National Key R\&D Program of China (No. 2020YFA0712700), the National Natural Science Foundation of China (Nos. 11931004, 12090010, 12090014) and the Youth Innovation Promotion Association, CAS (Y2021002). The second author is grateful to the National Natural Science Foundation of China (No.12231002). The third author is grateful to the National Natural Science Foundation of China (Nos. 12288201, 12271352, 12201611).

\end{document}